\documentclass[10pt,letterpaper,psamsfonts]{amsart}
\usepackage[utf8x]{inputenc}
\usepackage[T1]{fontenc}
\usepackage[american]{babel}

\usepackage[dvipsnames]{xcolor}

\usepackage{mathrsfs}
\usepackage{amsfonts}
\usepackage{amsmath}
\usepackage{amssymb}
\usepackage{amsthm}
\usepackage{stmaryrd}
\usepackage{amscd}
\usepackage{subfig}
\usepackage{booktabs}

\usepackage{tikz}
\usetikzlibrary
  {arrows,calc,through,backgrounds,matrix,positioning,decorations.pathmorphing}
\usepackage{gnuplot-lua-tikz} % TODO: Convert final plots and remove package

\theoremstyle{plain}

\newtheorem{lemma}{Lemma}

\theoremstyle{remark}
\newtheorem{remark}{Remark}

\usepackage{mathtools} % Bonus

\usepackage{textcomp}
\usepackage{yfonts}

\usepackage{color}

\usepackage{enumerate}

\newcommand{\vertiii}[1]{{\left\vert\kern-0.25ex\left\vert\kern-0.25ex\left\vert #1 \right\vert\kern-0.25ex\right\vert\kern-0.25ex\right\vert}}

\newcommand{\inv}{{-1}}

\newcommand{\Jacobian}{\operatorname{D}}

\newcommand{\dif}{{\mathrm d}}

\newcommand{\grad}{\nabla}

\newcommand{\Laplace}{\Delta}

\newcommand{\divergence}{\operatorname{div}}

\newcommand{\Id}{\operatorname{Id}}

\newcommand{\trace}{\operatorname{tr}}

\newcommand{\supp}{\operatorname{supp}}
\newcommand{\RT}{{\bfR \bfT}}

\newcommand{\onehalf}{{\frac{1}{2}}}

\usepackage{nicefrac}
\newcommand{\niceonehalf}{{\nicefrac{1}{2}}}

\newcommand{\bbN}{{\mathbb N}}

\newcommand{\bbR}{{\mathbb R}}

\newcommand{\calF}{{\mathcal F}}

\newcommand{\calN}{{\mathcal N}}

\newcommand{\bfH}{{\mathbf H}}
\newcommand{\bfI}{{\mathbf I}}

\newcommand{\bfL}{{\mathbf L}}

\newcommand{\bfQ}{{\mathbf Q}}
\newcommand{\bfR}{{\mathbf R}}

\newcommand{\bfT}{{\mathbf T}}

\newcommand{\bfX}{{\mathbf X}}

\newcommand{\scrH}{{\mathscr H}}

\title[Flux Reconstruction for Goal-Oriented A Posteriori Estimation]{Flux Reconstruction for\\Goal-Oriented A Posteriori Estimation}

\author{Martin Licht}
\address{UCSD Department of Mathematics, %
9500 Gilman Drive MC0112, %
La Jolla, CA 92093-0112, USA}
\email{mlicht@ucsd.edu}
\thanks{This research was supported by the European Research Council through the FP7-IDEAS-ERC Starting
Grant scheme, project 278011 STUCCOFIELDS}
\author{Matthias Maier}
\address{School of Mathematics, University of Minnesota, %
Minneapolis, Minnesota 55455, USA}
\email{msmaier@umn.edu}



\allowdisplaybreaks

\begin{document}

%%%%%%%%%%%%%%%%%%%%%%%%%%%%%%%%%%%%%%%%%%%%%%%%%%%%%%%%%%%
%%%%%%%%%%%%% ABSTRACT %%%%%%%%%%%%%%%%%%%%%%%%%%%%%%%%%%%%
%%%%%%%%%%%%%%%%%%%%%%%%%%%%%%%%%%%%%%%%%%%%%%%%%%%%%%%%%%%

 % What was done?
 % Why did you do it?
 % What did you find?
 % Why are these findings useful and important?

\begin{abstract}
 We propose a new heuristic goal-oriented a posteriori error estimator 
 that connects the dual weighted residual method 
 with equilibrated a posteriori error estimation. 
 Our numerical experiments demonstrate the practical reliability 
 of the error estimator, confirming theoretical predictions,
 as well as optimally convergent adaptivity 
 even over singular domains and coarse meshes.  
 The central algorithm is a localized flux reconstruction, 
 which has been implemented in the finite element library deal.II.
 For a solid preparation we assess the performance 
 of the equilibrated a posteriori error estimator 
 of the energy norm 
 in numerical experiments. 
 Moreover, we give what seems to be first rigorous discussion 
 in the numerical literature 
 of localized flux reconstruction over quadrilateral meshes 
 with hanging nodes. 
\end{abstract}

\maketitle

%%%%%%%%%%%%%%%%%%%%%%%%%%%%%%%%%%%%%%%%%%%%%%%%%%%%%%%%%%%
%%%%%%%%%%%%% CONTENT %%%%%%%%%%%%%%%%%%%%%%%%%%%%%%%%%%%%%
%%%%%%%%%%%%%%%%%%%%%%%%%%%%%%%%%%%%%%%%%%%%%%%%%%%%%%%%%%%

\section{Introduction}
\label{sec:introduction}

%%%%%%%%% 
%%%%%%%%% INTRODUCTION 
%%%%%%%%% 

%
% FIRST PARAGRAPH 
% 
% - tell something about a posteriori estimation in general
% - briefly tell about equilibrated error estimators
% - then focus on local divergence equations over patches
% - mention hypercircle identity 
% - mention intricate auxiliary computations 
%
\emph{A posteriori error estimation} is an important concept 
in assessing the accuracy of finite element methods for partial differential equations. 
% In addition to providing global error estimates, 
Many a posteriori error estimators not only bound the global error 
but also indicate the local contributions of the approximation error.
Thus, they are fundamental to \emph{adaptive finite element methods}
based on local mesh refinement. 
The overall importance of a posteriori estimation is reflected by
the large corpus of literature on this topic
(see \cite{ainsworth2011posteriori,repin2008posteriori,verfurth2013posteriori} 
and the references therein).
Typically, convergence is shown with respect to global energy norms. 
In many applications, however, we are more
interested in approximating a \emph{quantity of interest},
assumed (for simplicity) to be a linear functional of the true solution. 
\emph{Goal-oriented a posteriori error estimates} 
bound or approximate the error in the quantity of interest. 

\emph{Equilibrated error estimators} are provably reliable and are considered
to be among the most efficient residual-based error estimators
\cite{ainsworth2007analysis}. 
They come in several variants, whose common idea is to compute error
estimates by solving localized subproblems with data constructed from
the global right-hand sides and the approximate solution. In this article,
we focus on the family of equilibrated error estimators based on solving
local divergence equations over patches around mesh nodes (see also
\cite{kelly1984self,ladeveze1983error,luce2004local,braess2008equilibrated,braess2009equilibrated,carstensen2013effective,kim2015postprocessing}).
These error estimators are reliable, constant-free, and computable. 
% % % \\

The contributions of this article 
are both theoretical and experimental in nature. 
We report on an  
equilibrated a posteriori error estimators 
implemented in the finite element software library deal.II
\cite{dealII85}.
Our experiments with the Poisson problem 
show a typical overestimation of the energy error 
by 30--70\% and optimal convergence 
of adaptive finite element methods, 
in accordance with similar findings 
in the numerical literature \cite{carstensen2013effective}.

An innovation at this point 
is our exposition of the construction 
and well-posedness of the local divergence equations  
in the finite element flux reconstruction;
see Lemma~\ref{lemma:einzigeslemma}. 
To the best of our knowledge, 
we give the first such account 
for the case of quadrilateral meshes with hanging nodes.
Since deal.II and many other finite element software libraries 
employ this class of meshes, 
this article closes a practically relevant gap in the literature. 
We also remark that all our results are stated 
for mixed boundary conditions 
and in arbitrary dimension.
% % % \\

These theoretical and practical results 
prepare our successive contributions to goal-oriented 
a posteriori error estimation. 
% This is a theoretical and practical preparation 
% for new results on goal-oriented a posteriori error estimation. 
We develop and assess a new heuristic 
goal-oriented error estimator, 
combining the basic idea of the dual weighted residual method 
\cite{becker2001optimal}
with techniques of equilibrated a posteriori error estimation. 

Our new error estimator requires finite element approximations 
for the original (primal) problem 
and the adjoint (dual) problem 
associated with the goal functional. 
In addition, it needs localized flux reconstructions for both problems. 
We show that our error estimator coincides with the true error 
in the quantity of interest up to a perturbance 
that depends on approximation errors of mixed finite element methods 
for the primal and the dual problem.
The latter term typically converges with higher order 
than the error estimator, depending on the regularity of the domain.
This is in accordance with our computational experiments, 
which demonstrate the practical reliability of the new estimator 
both on regular and singular domains. 
We furthermore observe optimal convergence of a corresponding 
goal-oriented adaptive finite element method. 

Computational experiments compare our error estimator 
with a dual weighted residual estimator 
and a goal-oriented error estimator proposed 
by Mozolevski and Prudhomme \cite{mozolevski2015goal}. 
The former leads to an optimally convergent 
goal-oriented adaptive finite element method 
but significantly underestimates the error 
on coarse triangulations 
and on singular domains. 
The latter gives reliable error estimates in practice 
but displays suboptimal adaptive convergence rates 
in our numerical experiments, 
which we attribute to oscillatory behavior. 
Our new error estimator consistently avoids those problems. 
The remainder of this article is structured as follows. 
In Section~\ref{sec:background} we review the analytical background. 
In Section~\ref{sec:hypercircle:poisson} we give a concise introduction to the
method of the hypercircle for the Poisson problem. 
Subsequently, Section~\ref{sec:flux} describes the localized flux reconstruction. 
Section~\ref{sec:hypercircle:qoi} discusses goal-oriented a posteriori error estimation. 
Finally, Section~\ref{sec:experiments} presents a number of numerical experiments. 
A conclusive summary follows in Section~\ref{sec:conclusion}.

\section{Analytical Background}
\label{sec:background}
% \mm{The analytical background is very thorough. Is it possible to shorten
% it slightly?}

% AGENDA OF THIS SECTION
% 
%  - Outline of geometry and the function spaces 
%  - Poincare-Friedrichs inequalities
%    

In this section we recall a number of basic function spaces and some fundamental results.
Throughout the paper we let $\Omega \subset \bbR^{n}$ be a fixed bounded Lipschitz domain. 
Furthermore, we fix a pair $\Gamma_{D}$ and $\Gamma_{N}$ of relatively open subsets of $\partial\Omega$ 
such that 
$\partial\Omega = \overline\Gamma_{D} \cup \overline\Gamma_{N}$
and 
$\Gamma_{D} \cap \Gamma_{N} = \emptyset$.

Let $L^{p}(\Omega)$ denote the Lebesgue space over $\Omega$ with exponent $p \in [1,\infty]$. 
These spaces are equipped with the canonical $L^{p}$ norms $\|\cdot\|_{L^{p}}$.
In the case $p = 2$, this norm is induced by a canonical scalar product $\langle\cdot,\cdot\rangle_{L^{2}}$.
Furthermore, $\bfL^{p}(\Omega) = L^{p}(\Omega)^{n}$ denotes the Banach space of vector fields over $\Omega$ with coefficients in $L^{p}(\Omega)$. 
This space is equipped with the canonical norm $\|\cdot\|_{\bfL^{p}}$.
Again, this norm is induced by a scalar product $\langle\cdot,\cdot\rangle_{\bfL^{2}}$ in the case $p=2$.

% We may consider more general norms on the vector spaces. 
We call a matrix field $A \in L^{\infty}(\Omega)^{n \times n}$ an \emph{admissible metric tensor} 
if $A$ is symmetric and invertible almost everywhere over $\Omega$ 
with $A^{\inv} \in L^{\infty}(\Omega)^{n \times n}$. 
% essentially bounded condition numbers. 
% such that the condition numbers are essentially bounded over $\Omega$. 
Every such $A$ induces a bounded isomorphism $A : \bfL^{2}(\Omega) \rightarrow \bfL^{2}(\Omega)$
by multiplication
and induces the $A$-scalar product
$\langle \sigma, \tau \rangle_{A} := \langle \sigma, \tau \rangle_{\bfL^{2}_{A}(\Omega)} := \langle \sigma, A \tau \rangle_{\bfL^{2}}$
over $\bfL^{2}(\Omega)$, 
which is equivalent to the canonical scalar product on $\bfL^{2}(\Omega)$.

We write $H^{1}(\Omega)$ for the first-order Sobolev space over $\Omega$, 
% i.e.\ the members of $L^{2}(\Omega)$ whose gradients 
% (a priori taken in the sense of distributions) are in $\bfL^{2}(\Omega)$. 
and we let $\bfH(\Omega,\divergence)$ denote the subspace of 
$\bfL^{2}(\Omega)$ whose members have their divergences (a priori taken in the sense of distributions) in $L^{2}(\Omega)$. 
We equip $H^{1}(\Omega)$ and $\bfH(\Omega,\divergence)$ with the canonical scalar products. 
% The canonical scalar products are 
% % These spaces are equipped with the canonical scalar products 
% \begin{gather}
%  \label{math:scalarproduct:H1}
%  \langle u, v \rangle_{H^{1}}
%  :=
%  \langle u, v \rangle_{L^{2}}
%  +
%  \langle \grad u, \grad v \rangle_{\bfL^{2}},
%  \quad
%  u, v \in H^{1}(\Omega)
%  ,
%  \\
%  \label{math:scalarproduct:Hdiv}
%  \langle \sigma, \tau \rangle_{\bfH(\divergence)}
%  :=
%  \langle \sigma, \tau \rangle_{\bfL^{2}} 
%  +
%  \langle \divergence \sigma, \divergence \tau \rangle_{L^{2}},
%  \quad 
%  \sigma, \tau \in \bfH(\Omega,\divergence)
%  .
% \end{gather}
We let $H^{1}(\Omega,\Gamma_{D})$ be the closed subspace of $H^{1}(\Omega)$
whose members have vanishing trace along $\Gamma_{D}$.
% {\matthias
Similarly, $\bfH(\Omega,\Gamma_{N},\divergence)$ shall denote the closed
subspace of
% }%
$\bfH(\Omega,\divergence)$ whose members have vanishing normal trace along
$\Gamma_{N}$. The integration-by-parts formula 
\begin{gather}
 \label{math:integrationbyparts}
 \int_{\Omega} \langle \grad v, \tau \rangle \;\dif x
 + 
 \int_{\Omega} \langle v, \divergence \tau \rangle \;\dif x
 =
 \oint_{\partial\Omega} \trace v \cdot \trace_{N} \tau \;\dif s
\end{gather}
holds for every $v \in H^{1}(\Omega)$ and $\tau \in \bfH(\Omega,\divergence)$.
Here, the boundary integral is understood in the generalized sense 
and vanishes if $v \in H^{1}(\Omega,\Gamma_{D})$ and $\tau \in \bfH(\Omega,\Gamma_{N},\divergence)$. 
% % % The following closed, densely-defined, unbounded operators are mutually adjoint:
% % % \begin{gather}
% % %  \label{math:denselydefineddifferentials:grad}
% % %  \grad :
% % %  H^{1}(\Omega,\Gamma_{D}) \subset L^{2}(\Omega)
% % %  \rightarrow
% % %  \bfL^{2}(\Omega),
% % %  \\
% % %  \label{math:denselydefineddifferentials:div}
% % %  - \divergence :
% % %  \bfH(\Omega,\Gamma_{N},\divergence) \subset \bfL^{2}(\Omega)
% % %  \rightarrow
% % %  L^{2}(\Omega).
% % % \end{gather}
We let $H^{-1}(\Omega,\Gamma_{N}) := H^{1}(\Omega,\Gamma_{D})^{\ast}$
denote the topological dual space of $H^{1}(\Omega,\Gamma_{D})$.
% {\matthias
This is again a Hilbert space equipped with the canonical operator norm.
% }%
We also write $\langle\cdot,\cdot\rangle$ for the distributional pairing
between $H^{1}(\Omega,\Gamma_{D})$ and $H^{-1}(\Omega,\Gamma_{N})$. We have
a bounded operator
\begin{gather}
 \label{math:boundeddifferential:grad}
 \grad : 
 H^{1}(\Omega,\Gamma_{D})
 \rightarrow \bfL^{2}_{A}(\Omega),
 \quad 
 v \mapsto \grad v
 ,
\end{gather}
whose dual is given by the bounded operator 
\begin{gather}
 \label{math:boundeddifferential:div}
 - \divergence_{\Omega,\Gamma_{N}} : 
 \bfL^{2}_{A}(\Omega)
 \rightarrow 
 H^{-1}(\Omega,\Gamma_{N}),
 \quad 
 \tau \mapsto \langle \tau, \grad \cdot \rangle_{A} 
 .
\end{gather}
This extension of the divergence operator to $\bfL^{2}(\Omega)$ 
% is natural in the sense that it 
commutes with the natural embeddings.
We have the generalized integration-by-parts formula 
\begin{gather}
 \label{math:integrationbyparts:distributional}
 \langle \tau, \grad v \rangle_{A}
 = 
 \langle \divergence_{\Omega,\Gamma_{N}} \tau, v \rangle,
 \quad 
 \tau \in \bfL^{2}(\Omega), \quad v \in H^{1}(\Omega,\Gamma_{D}).
\end{gather}
For our discussion of the Poisson problem, 
we let $\scrH(\Omega,\Gamma_{D})$ denote the locally constant functions contained in $H^{1}(\Omega,\Gamma_{D})$.
This is just the span of the indicator functions of those connected components of $\Omega$
that do not touch $\Gamma_{D}$. 
% With some abuse of notation, 
We let $\scrH_{\Omega,\Gamma_{D}} v$ denote the $L^{2}$-orthogonal projection 
of any $v \in L^{2}(\Omega)$ onto $\scrH(\Omega,\Gamma_{D})$. 
The linear mapping $\scrH_{\Omega,\Gamma_{D}}$ extends 
naturally to $H^{\inv}_{N}(\Omega)$.

\begin{remark}
 On a connected domain, 
 the space $\scrH(\Omega,\Gamma_{D})$ will be 
 either the span of the constant function or trivial, 
 depending on whether $\Gamma_{D}$ is empty or not. 
 This notation will come in handy later in Section~\ref{sec:flux}, 
 where we consider local divergence equations 
 over varying subdomains of $\Omega$ with varying boundary conditions.
%  on the one hand, pure Neumann boundary conditions yield a space that is spanned by the constant functions,
%  and on the other hand, pure Dirichlet or mixed Dirichlet-Neumann boundary conditions
%  yield a space that is trivial. 
%  on the one hand, pure Neumann boundary conditions, 
%  where $\scrH(\Omega,\Gamma_{D}) \simeq \bbR$, 
%  and on the other hand, pure Dirichlet or mixed Dirichlet-Neumann boundary conditions, 
%  where $\scrH(\Omega,\Gamma_{D}) = \{0\}$.
\end{remark}

With these definitions in mind, we state the following
Poincar\'e-Friedrichs inequalities. There exists a constant $C^{\rm PF}_{A} >
0$, depending only on $\Omega$, $\Gamma_{D}$ and $A$, such that for every $v \in
H^{1}(\Omega,\Gamma_{D})$ we have 
\begin{gather}
 \label{math:poincarefriedrichs:grad}
 \left\| v - \scrH_{\Omega,\Gamma_{D}} v \right\|_{L^{2}} 
 \leq 
 C^{\rm PF}_{A} \| \grad v \|_{A}
%  + \| \scrH_{\Omega,\Gamma_{D}} v \|_{L^{2}}
 .
\end{gather}
On the other hand, 
% % % for every $f \in L^{2}(\Omega)$ there exists $\tau \in
% % % \bfH(\Omega,\Gamma_{N},\divergence)$ such that 
% % % \begin{gather}
% % %  \label{math:poincarefriedrichs:div}
% % %  - \divergence \tau = f - \scrH_{\Omega,\Gamma_{D}} f,
% % %  \qquad 
% % %  \| \tau \|_{A}
% % %  \leq
% % %  C^{\rm PF}_{A}
% % %  \left\| f - \scrH_{\Omega,\Gamma_{D}} f \right\|_{L^{2}} 
% % % %  \left( 
% % % %  \| f \|_{L^{2}}^{2} 
% % % %  - 
% % % %  \| \scrH_{\Omega,\Gamma_{D}} f \|^{2}_{L^{2}}
% % % %  \right)
% % %  .
% % % \end{gather}
% % % This existence result and the inequality extend naturally to $H^{-1}(\Omega,\Gamma_{N})$.
for every $F \in H^{-1}(\Omega,\Gamma_{N})$ there exists $\tau \in \bfL^{2}(\Omega)$ 
such that 
\begin{gather}
 \label{math:poincarefriedrichs:distdiv}
 - \divergence_{\Omega,\Gamma_{N}} \tau = F - \scrH_{\Omega,\Gamma_{D}} F,
 \quad
 \| \tau \|_{A}
 \leq
 C^{\rm PF}_{A} 
 \| F - \scrH_{\Omega,\Gamma_{D}} F \|_{H^{-1}(\Omega,\Gamma_N)} 
 .
\end{gather}
Note that $F \in L^{2}(\Omega)$ implies $\tau \in \bfH(\Omega,\Gamma_{N},\divergence)$
in \eqref{math:poincarefriedrichs:distdiv}. 
% % % Consequently, the gradient operator \eqref{math:denselydefineddifferentials:grad}
% % % % , the divergence operator \eqref{math:denselydefineddifferentials:div}, 
% % % and the distributional divergence \eqref{math:boundeddifferential:div}
% % % have closed range and bounded generalized inverses. 
% $\grad : H^{1}(\Omega,\Gamma_{D}) \rightarrow \bfL^{2}(\Omega)$,
% the divergence operator $\divergence : H(\Omega,\Gamma_{N},\divergence) \rightarrow L^{2}(\Omega)$,
% and the distributional divergence $\divergence_{\Omega,\Gamma_{N}} : \bfL^{2}(\Omega) \rightarrow H^{-1}(\Omega,\Gamma_{N})$
% have closed range and bounded generalized inverses. 
% % % Another implication is that 
% % % % Due to the Poincar\'e-Friedrichs inequalities \eqref{math:poincarefriedrichs:div} and \eqref{math:poincarefriedrichs:distdiv}
% % % every $v \in H^{1}(\Omega,\Gamma_{D})$ satisfies the estimate 
% % % \begin{gather*}
% % %  \| \grad v \|_{A} 
% % %  \leq 
% % %  C^{\rm PF}_{A} 
% % %  \| \divergence_{\Omega,\Gamma_{N}} A \grad v \|_{H^{-1}(\Omega,\Gamma_N)}, 
% % % \end{gather*}
% % % and if additionally even $A \grad v \in \bfH(\Omega,\Gamma_{N},\divergence)$, 
% % % then 
% % % \begin{gather*}
% % %  \| \grad v \|_{A} \leq C^{\rm PF}_{A} \| \divergence_{\Omega,\Gamma_{N}} A \grad v \|_{L^{2}}.
% % % \end{gather*}
% % % % \mm{Q: Some $\scrH(\Omega,\Gamma_{D})$ projection missing? What about the counterexample
% % % % $\Gamma_{D}=\emptyset$ and $v=x_1$?}

Finally, a modicum of Hodge theory is utilized throughout this article. 
We set 
\begin{align}
 \label{math:hodgecomponent}
 \begin{split}
  \bfX(\Omega,\Gamma_{N},A)
  &:= 
  \grad \left( H^{1}(\Omega,\Gamma_{D}) \right)^{\perp_{A}}
  =
  \left\{ \tau \in \bfL^{2}(\Omega) : \divergence_{\Omega,\Gamma_{N}} A \tau = 0 \right\}
%   \\&=
%   \ker\left( \divergence_{\Omega,\Gamma_{N}} A : \bfL^{2}(\Omega) \rightarrow H^{-1}(\Omega,\Gamma_{N}) \right)
  .
 \end{split}
\end{align}
This is precisely the Hilbert space of those $\tau \in \bfL^{2}(\Omega)$ 
for which $A\tau \in \bfH(\Omega,\divergence)$ 
has vanishing divergence and homogeneous normal trace along $\Gamma_{N}$.
Basic functional analysis gives the $A$-orthogonal \emph{Hodge-Helmholtz decomposition} 
\begin{gather}
 \label{math:hodgedecomposition}
 \bfL^{2}(\Omega) = \grad H^{1}(\Omega,\Gamma_{D}) \oplus_{A} \bfX(\Omega,\Gamma_{N},A)
 .
\end{gather}
% % % follows by basic functional analysis.

\section{Error Estimates for the Poisson Problem}
\label{sec:hypercircle:poisson}

In this section we outline reliable error estimates for approximate
solutions of the Poisson problem based on the hypercircle identity
and its generalizations.
Here, our results rely on techniques in functional analysis,  
and no further details on the method of approximation 
are assumed at this point.

\subsection{Model Problem}
\label{subsec:hypercircle:poisson:modelproblem}

Our object of discussion is the partial differential equation 
\begin{gather}
 \label{math:poisson}
 - \divergence_{\Omega,\Gamma_{N}} A \grad u = F - \scrH_{\Omega,\Gamma_{D}} F,
 \quad 
 u \perp \scrH(\Omega,\Gamma_{D}), 
\end{gather}
where $F \in H^{-1}(\Omega,\Gamma_{N})$ is the data,
$A \in L^{\infty}(\Omega)^{n \times n}$ is a fixed admissible metric tensor, 
and the function $u \in H^{1}(\Omega,\Gamma_{D})$ is the unknown.
The Poisson equation in above form is precisely the weak formulation 
that characterizes $u$ by requiring 
\begin{subequations}
\label{math:poisson:weakformulation}
\begin{gather}
 \langle \grad u, \grad v \rangle_{A} = F\left( v - \scrH_{\Omega,\Gamma_{D}} v \right),
 \quad 
 v \in H^{1}(\Omega,\Gamma_{D}),
 \\
 u \perp \scrH(\Omega,\Gamma_{D}).
\end{gather}
\end{subequations}
The well-posedness of this problem 
% \begin{matthias}
follows from standard elliptic regularity theory.
% \end{matthias}
For every right-hand side $F
\in H^{-1}(\Omega,\Gamma_{N})$ there exists a unique solution $u \in
H^{1}(\Omega,\Gamma_{D})$ of \eqref{math:poisson}, and we can estimate 
\begin{gather}
 \label{math:poisson:stability}
 \| u \|_{L^{2}}
 \leq 
 \left(C^{\rm PF}_{A}\right)^{2}
 \| F \|_{H^{-1}(\Omega,\Gamma_{N})},
 \quad 
 \| \grad u \|_{A}
 \leq 
 C^{\rm PF}_{A} 
 \| F \|_{H^{-1}(\Omega,\Gamma_{N})}
 .
%  \divergence_{\Omega,\Gamma_{N}} A \grad u = F - \scrH_{\Omega,\Gamma_{D}} F,
%  \qquad 
%  u \perp \scrH(\Omega,\Gamma_{D}).
\end{gather}
Throughout this section,  
a candidate approximation $u_{h} \in H^{1}(\Omega,\Gamma_{D})$ 
is assumed to be already known.
We want to compute information about the error $u - u_{h}$, 
such as bounds in Sobolev norms.
In our prospective applications, 
$u_{h}$ is computed with a Galerkin method for the Poisson problem,
but the exact solution $u$ is unknown.

\begin{remark}
  The Poisson equation is often studied only for square-integrable
  right-hand sides $F \in L^{2}(\Omega)$, so that $\grad u \in
  H(\Omega,\Gamma_{N},\divergence)$ by definition. We consider a
  distributional right-hand side in $H^{-1}(\Omega,\Gamma_{N})$ in this
  article, which allows for a neat formalism for inhomogeneous mixed
  boundary conditions.
\end{remark}

\subsection{Reliable Error Estimation}
\label{subsec:hypercircle:poisson:reliability}

We derive reliable error estimates in the energy seminorm via the
\emph{hypercircle method}. For that purpose we assume that $\sigma_{h} \in
\bfL^{2}(\Omega)$ is \emph{any} solution of the \emph{flux equation} 
\begin{gather}
 \label{math:fluxequation:discrete}
 - \divergence_{\Omega,\Gamma_{N}} A \sigma_{h} = F - \scrH_{\Omega,\Gamma_{D}} F.
\end{gather}
The availability of such $\sigma_{h}$ is stipulated for the time being.
The specific auxiliary computations 
that solve the flux equation \eqref{math:fluxequation:discrete} in applications
will be examined later in this article.
% Moreover,
% the more general case of a merely approximate solution to the flux equation is discussed below,
% and specific means of solving the flux equation \eqref{math:fluxequation:discrete}
% in applications will be examined later in this article. 
% Utilizing this
% technique requires additional data, which in applications must be obtained
% by auxiliary computations. 

Under this assumption, it is elementary to derive 
\begin{align*}
 \| \grad u_h - \sigma_h \|^{2}_{A}
 &=
 \| \grad u_h - \grad u \|^{2}_{A}
 +
 \| \grad u   - \sigma_h \|^{2}_{A}
%  \\
%  &\qquad\qquad
 +
 2\left\langle
  \grad \left( u_h - u \right),
  A \nabla u - A \sigma_h
 \right\rangle_{\bfL^{2}}
 .
\end{align*}
Since $\sigma_{h}$ solves the flux equation \eqref{math:fluxequation:discrete},
integration by parts yields
\begin{align*}
 \left\langle \grad \left( u_h - u \right), A \grad u - A \sigma_h \right\rangle_{\bfL^{2}}
 &=
 \left\langle u_h - u, \divergence_{\Omega,\Gamma_{N}} A \grad u - \divergence_{\Omega,\Gamma_{N}} A \sigma_h \right\rangle_{\bfL^{2}}
 \\
 &=
 \left\langle u_h - u, ( \Id - \scrH_{\Omega,\Gamma_{D}}) F - ( \Id - \scrH_{\Omega,\Gamma_{D}}) F \right\rangle
 \\
 &=
 0
 . 
\end{align*}
Thus we obtain what is known as the \emph{hypercircle identity}:
\begin{align}
 \label{math:hypercircle:slim}
 \| \grad u_h - \sigma_h \|^{2}_{A}
 &=
 \| \grad u_h - \grad u \|^{2}_{A}
 +
 \| \grad u   - \sigma_h \|^{2}_{A}
 .
\end{align}
We note that the left-hand side of \eqref{math:hypercircle:slim} 
is given explicitly in terms of entities that are computable by assumption. 
We thus get the simple and computable estimate
\begin{align}
 \label{math:simpleestimate}
 \| \grad u_h - \sigma_h \|_{A}
 &\geq 
 \| \grad u_{h}  - \grad u \|_{A}
 .
\end{align}
This bounds the error $u - u_{h}$ in the energy seminorm, 
which is the dominant term of the full $H^{1}$ error norm in typical applications, 
in a manner computable in terms of only $\grad u_{h}$ and $\sigma_{h}$.

\begin{remark}
 Any computable estimate for $\grad u - \grad u_{h}$ in the $A$-norm 
 gives an estimate for $u - u_{h}$ in the $L^{2}$-norm, 
 as follows by the Poincar\'e-Friedrichs inequality \eqref{math:poincarefriedrichs:grad}.
 This estimate is computable 
 if a bound for the Poincar\'e-Friedrichs constant $C^{\rm PF}_{A}$ can be computed.
 Such estimates are not within the scope of this article, 
 but we refer to the literature \cite{pauly2015maxwell} for research in that direction.
\end{remark}

% \begin{remark}
%  \label{remark:thales}
%  Since $\grad u_h - \grad u$ and $\grad u - \sigma_h$ are orthogonal in $A$, 
%  the points $\grad u$, $\grad u_h$ and $\sigma_h$ build a triangle 
%  with an $A$-orthogonal angle at $\grad u$ 
%  and the hypotenuse being the line segment between $\grad u_h$ and $\sigma_h$.
%  The Pythagorean theorem is just the error identity \eqref{math:hypercircle:slim}.
%  On the other hand, Thales' theorem implies that
%  $\grad u$, $\grad u_h$, and $\sigma_h$ are all located on a circle 
%  whose center $\omega_h$ and radius $t_h$ are given by
%  \begin{gather}
%   \label{math:thales:centerandradius}
%   \omega_h := \onehalf \grad u_h + \onehalf \sigma_h,
%   \quad 
%   t_h
%   :=
%   \onehalf \| \grad u_h - \sigma_h \|_{A}.
%  \end{gather}
%  We note as an easy consequence that $\omega_h$,
%  the average of the discrete gradient $\grad u_h$ and the flux variable $\sigma_h$,
%  is a computable approximation to $\grad u$ 
%  with exactly computable error $t_h$ with respect to the $A$-norm. 
% \end{remark}

\begin{remark}
 The error identity \eqref{math:hypercircle:slim} goes back to the seminal research 
 of Prager and Synge in the field of mathematical elasticity \cite{prager1947approximations}
 and is thus known as \emph{Prager-Synge identity} in the literature. 
 The underlying technique is also called \emph{hypercircle method} 
 because, as a consequence of Thales' theorem in Hilbert spaces, 
 the vector fields $\grad u$, $\grad u_{h}$, and $\sigma_{h}$
 are located on a common hypercircle in the space of square-integrable vector fields
 \cite{kikuchi2007remarks,vejchodsky2006guaranteed,vejchodsky2004local}.
 Furthermore, the technique is known as \emph{two-energy principle}
 because it compares an approximate minimizer of the Poisson energy 
 to an approximate minimizer of its conjugate energy \cite{braess2014equilibrated}.
 In the context of a posteriori error estimates,
 Equation~\eqref{math:simpleestimate} is called a \emph{constant-free}
 error estimate because it does not involve generic or uncertain constants,
 unlike, say, the classical residual error estimator
 \cite{verfurth2009note}. We also bring to attention that
 \emph{functional-type error estimates} \cite{pauly2009functional} utilize
 similar techniques.
\end{remark}

\subsection{Efficiency}
\label{subsec:hypercircle:poisson:efficiency}

Having established the reliable error estimate \eqref{math:simpleestimate},
we now discuss its efficiency.
The latter depends, broadly speaking, on the ``efficiency'' of 
% The efficiency
% of the error estimate depends, broadly speaking, on the ``efficiency'' of
$\sigma_{h}$ as a solution to the flux equation.
We first recall that 
\begin{gather*}
 \divergence_{\Omega,\Gamma_{N}} A ( \grad u - \sigma_{h} ) = F - F = 0
 .
\end{gather*}
Now $\grad u - \sigma_{h} \in \bfX(\Omega,\Gamma_{N},A)$
follows from the the definition of $\bfX(\Omega,\Gamma_{N},A)$
in Equation \eqref{math:hodgecomponent}. 
Thus there exists a unique $\theta_{h} := \grad u - \sigma_{h} \in \bfX(\Omega,\Gamma_{N},A)$ 
such that we have the $A$-orthogonal decomposition 
\begin{gather}
 \label{math:flux:hodgedecomposition}
 \sigma_{h} = \grad u + \theta_{h}.
\end{gather}
Via the Pythagorean theorem, we readily compute that 
\begin{gather*}
 \| \grad u_h - \sigma_h \|^{2}_{A}
 =
 \| \grad u_h - \grad u - \theta_{h} \|^{2}_{A}
 =
 \| \grad u_h - \grad u \|_{A}^{2} + \| \theta_{h} \|^{2}_{A}
 .
\end{gather*}
We quantify the efficiency of the error estimate \eqref{math:simpleestimate} as
\begin{gather}
 \label{math:efficiency}
 \dfrac{ \| \grad u_h - \sigma_h \|^{2}_{A} }{ \| \grad u_h - \grad u \|^{2}_{A} }
 =
 1
 +
 \dfrac{ \| \theta_h \|^{2}_{A} }{ \| \grad u_h - \grad u \|^{2}_{A} }
 .
\end{gather}
% \begin{martin}
We conclude that the ``efficiency'' of the flux reconstruction
determines the efficiency of the error estimate:
the rigorously reliable error estimate \eqref{math:simpleestimate} 
will be the sharper, the smaller the norm of $\theta_{h} = \sigma_{h} - \nabla u$.
% , i.e., the direct summand of $\sigma_{h}$ in $\bfX(\Omega,\Gamma_{N},A)$. 
% \end{martin}%
This generally depends on the specific construction of $\sigma_{h}$ in applications;
for a specific construction,
a bound of $\| \theta_h \|_{A}$ in terms of $\| \grad u_h - \grad u \|_{A}$
will be given later in this article.
% We $\| \theta_h \|_{A}$
% The rigorously reliable error estimate \eqref{math:simpleestimate} will be
% \emph{efficient} only if we can establish a bound of $\| \theta_h \|_{A}$
% in terms of (a preferably small fraction of) $\| \grad u_h - \grad u \|_{A}$. 
% Preferably, the flux
% reconstruction should give $\| \theta_h \|_{A} \ll \| \grad u_h - \grad u
% \|_{A}$. 

\begin{remark}
 The relevance of the Hodge-Helmholtz decomposition 
 for equilibrated error estimation 
 is recognized in the published literature 
 \cite{carstensen2013effective}.
 Some simple postprocessing techniques are known to 
 effectively improve the approximation properties 
 of the numerically computed flux
 and thus the efficiency of the resulting 
 error estimates. 
%  
%  When applying the theory to finite element methods,
%  the local and the global efficiency of \eqref{math:simpleestimate}
%  can be shown by comparing the 
%  equilibrated error estimator to the 
%  classical explicit residual error estimator \cite{braess2008equilibrated}. % Section 2.4
%  This will be discussed later in this article.
\end{remark}

\subsection{Approximate Flux Reconstruction}
\label{subsec:hypercircle:poisson:reconstruction}

In many applications it may not be computationally feasible 
to construct a vector field $\sigma_{h} \in \bfL^{2}(\Omega)$ 
solving the flux equation \eqref{math:fluxequation:discrete} exactly. 
Instead, 
we may have a vector field $\sigma_{h} \in \bfL^{2}(\Omega)$ 
that solves an \emph{approximate flux equation} 
\begin{gather}
 \label{math:flux:equation:approximate}
 - \divergence_{\Omega,\Gamma_{N}} A \sigma_{h} = F_{h} - \scrH_{\Omega,\Gamma_{D}} F_{h}
\end{gather}
for some right-hand side $F_{h} \in H^{\inv}(\Omega,\Gamma_{N})$. 
In typical applications, $F_{h}$ is a given approximation of $F$ and the
variable $\sigma_{h}$ is constructed from $F_{h}$ as an approximate
solution of the original flux equation. 
% % % \\

For a reliable estimate of the $A$-norm of $\grad u - \grad u_{h}$, 
we formally introduce the unique solution $\widehat u \in H^{1}(\Omega,\Gamma_{D})$
of the \emph{approximate Poisson problem} 
\begin{gather*}
 - \divergence_{\Omega,\Gamma_{N}} A \grad \widehat u
 =
 F_{h} - \scrH_{\Omega,\Gamma_{D}} F_{h}
 ,
 \quad 
 \widehat u \perp \scrH(\Omega,\Gamma_{D})
 .
\end{gather*}
The triangle inequality states 
\begin{gather*}
 \left\| \grad u - \grad u_{h} \right\|_{A}
 \leq 
 \left\| \grad u - \grad \widehat u \right\|_{A}
 +
 \left\| \grad \widehat u - \grad u_{h} \right\|_{A}
 .
\end{gather*}
On the one hand, 
the Poincar\'e-Friedrichs inequality \eqref{math:poincarefriedrichs:distdiv} gives 
\begin{gather*}
 \left\| \grad u - \grad \widehat u \right\|_{A}
 \leq 
 C^{\rm PF}_{A}
 \left\| F - F_{h} \right\|_{H^{-1}(\Omega,\Gamma_{N})}
 . 
\end{gather*}
On the other hand,
when considering $u_{h}$ as an approximation to $\widehat u$,
the classical hypercircle identity gives 
\begin{gather*}
 \left\| \grad \widehat u - \grad u_{h} \right\|_{A}
 \leq 
 \left\| \sigma_{h} - \grad u_{h} \right\|_{A}
 .
\end{gather*}
In summary, we get the reliable error estimate 
\begin{gather}
 \label{math:reliableerrorestimate}
 \left\| \grad u - \grad u_{h} \right\|_{A}
 \leq 
 C^{\rm PF}_{A}
 \| F - F_{h} \|_{H^{-1}(\Omega,\Gamma_{N})}
 +
 \| \sigma_{h} - \grad u_{h} \|_{A}
 .
\end{gather}
% is a reliable error estimate. 
If $A \sigma_{h} \in H_{N}(\divergence,\Omega)$ and $F \in L^{2}(\Omega)$, 
then the negative Sobolev norms in \eqref{math:reliableerrorestimate} 
can be bounded by the $L^{2}$-norm of $F - F_{h}$.
The entire estimate is fully computable provided that 
an estimate for the Poincar\'e-Friedrichs constant is known.

Whereas this addresses the reliability, 
we also address the efficiency of \eqref{math:reliableerrorestimate}.  
This generally depends on the Hodge-Helmholtz decomposition 
of the flux variable $\sigma_{h}$, 
whose component in $\bfX(\Omega,\Gamma_{N},A)$ we desire to be as small as possible,
and on the consistency error $F_{h} - F_{h}$. 
By construction, 
there exists $\widehat \theta_{h} \in \bfX(\Omega,\Gamma_{N},A)$
with 
\begin{gather*}
 \widehat \theta_{h} = \sigma_{h} - \grad \widehat u.
\end{gather*}
The efficiency of the error estimate \eqref{math:reliableerrorestimate}
is the right-hand side of that inequality 
divided by the true error $\left\| \grad u - \grad u_{h} \right\|_{A}$.
With the observation 
\begin{align*}
 \| \sigma_{h} - \grad u_{h} \|_{A}
 &\leq 
 \| \widehat \theta_{h} \|_{A}
 +
 \| \grad \widehat u - \grad u \|_{A}
 +
 \| \grad u - \grad u_{h} \|_{A}
 \\
 &\leq 
 \| \widehat \theta_{h} \|_{A}
 +
 C^{\rm PF}_{A}
 \| F - F_{h} \|_{H^{\inv}(\Omega,\Gamma_{N})}
 +
 \| \grad u - \grad u_{h} \|_{A}
\end{align*}
we observe that the efficiency of \eqref{math:reliableerrorestimate} 
satisfies the computable bounds 
\begin{align}
 \label{math:efficiencyofestimate}
 \begin{split}
 1
 \leq\;
 &
 C^{\rm PF}_{A}
 \dfrac{
 \| F - F_{h} \|_{H^{-1}(\Omega,\Gamma_{N})}
 }{
 \| \grad u - \grad u_{h} \|_{A}
 }
 +
 \dfrac{
 \| \sigma_{h} - \grad u_{h} \|_{A}
 }{
 \| \grad u - \grad u_{h} \|_{A}
 }
 \\&\quad 
 \leq 
 1
 +
 2 C^{\rm PF}_{A}
 \dfrac{
 \| F - F_{h} \|_{H^{-1}(\Omega,\Gamma_{N})}
 }{
 \| \grad u - \grad u_{h} \|_{A}
 }
 +
 \dfrac{
 \| \widehat \theta_{h} \|_{A}
 }{
 \| \grad u - \grad u_{h} \|_{A}
 }
 .
 \end{split}
\end{align}
% Hence the efficiency of the error estimate requires that $F_{h}$
% is a good approximation of $F$
% and that $\sigma_{h}$ does not differ too much from $\grad \widehat u$. 
We rephrase this result in a suggestive manner:
the generalized estimate \eqref{math:efficiencyofestimate} is efficient provided that $\sigma_{h}$
is an efficient solution to an efficient approximation of the original flux equation.
Note that \eqref{math:efficiency} is recovered in the special case $F = F_{h}$.

\begin{remark}
 In many applications,
 not only the true right-hand side $F$ is approximated by an approximate right-hand side $F_{h}$,
 but also the true coefficient $A$ is approximated by an approximate coefficient $A_{h}$
 assumed to be an admissible metric tensor. 
 Suppose that $\sigma'_{h} \in \bfL^{2}(\Omega)$ solves the approximate flux equation 
 \begin{gather*}
  - \divergence_{\Omega,\Gamma_{N}} A_{h} \sigma'_{h} = F_{h} - \scrH_{\Omega,\Gamma_{D}} F_{h}
 \end{gather*}
 with the approximate coefficient $A_{h}$. 
 Via the simple observation $A_{h} = A \left( A^{\inv} A_{h} \right)$
 we can apply the error estimate \eqref{math:efficiencyofestimate}
 with the flux variable $\sigma_{h} = A^{\inv} A_{h} \sigma_{h}'$. 
 So the case of approximate coefficients 
 can be reduced to the case of exact coefficients. 
\end{remark}

\subsection{Residual Flux Reconstruction}
\label{subsec:hypercircle:poisson:residual}

In the next section, it will be conceptually helpful to use a variation 
of the flux equation \eqref{math:fluxequation:discrete} (or \eqref{math:flux:equation:approximate}, respectively).
Suppose that $F_{h} \in H^{-1}(\Omega,\Gamma_{N})$
is any approximate right-hand side.
We define the \emph{residual}
$r_{h} \in H^{-1}(\Omega,\Gamma_{N})$ by
\begin{gather}
 \label{math:residualdefinition}
 r_{h} = F_{h} - \scrH_{\Omega,\Gamma_{D}} F_{h} + \divergence_{\Omega,\Gamma_{N}} A \grad u_{h}.
\end{gather}
Given any solution $\varrho_{h} \in \bfL^{2}(\Omega)$
of the \emph{residual flux equation} 
\begin{gather}
 \label{math:residualfluxequation}
 - \divergence_{\Omega,\Gamma_{N}} A \varrho_{h} = r_{h},
\end{gather}
the vector field $\sigma_{h} = \varrho_{h} + \grad u_{h}$ 
solves the approximate flux equation \eqref{math:flux:equation:approximate}. 
% \begin{gather*}
%  - \divergence_{\Omega,\Gamma_{N}} A \sigma_{h} = F_{h} - \scrH_{\Omega,\Gamma_{D}} F_{h}.
% \end{gather*}
In many applications it is easier to first solve the residual flux equation \eqref{math:residualfluxequation} 
with right-hand side $r_{h}$ and derive $\sigma_{h}$ in the aforementioned manner, 
rather than directly solving the flux equation with the effective right-hand side $F_{h}$. 
This is intuitive: the approximate solution $u_{h}$ typically depends on the right-hand side $F_{h}$ 
and thus contains additional information on $F_{h}$ that can be used in the flux reconstruction. 

\section{Finite Element Flux Reconstruction}
\label{sec:flux}

In this section we discuss the implementation of a local flux reconstruction
over quadrilateral finite element meshes with hanging nodes.
In particular, we give a formal proof of the well-posedness 
of the local problems (see Lemma~\ref{lemma:einzigeslemma}),
and we keep the discussion independent of the spatial dimension. 
For the context and motivation of this section, 
we recall that the abstract error estimates of the previous section 
require a solution of the flux equation \eqref{math:fluxequation:discrete}, 
or equivalently, of the residual flux equation \eqref{math:residualfluxequation}. 
The localized flux reconstruction efficiently computes 
such a solution in finite element spaces. 

The main ideas of the local flux reconstruction can be found 
in the literature for other types of meshes \cite{braess2009equilibrated}. 
Classical finite element over quadrilateral non-conforming meshes 
are described in \cite{brezzi2012mixed}. 

\subsection{Finite Element Spaces}
\label{subsec:flux:fem}

% 
% Describe mesh combinatorics
% 
Let $\Omega_{h}$ be a partition of a polygonal domain $\Omega$
into convex, non-degenerate quadrilaterals. 
We relax the usual \emph{form regularity} 
by allowing \emph{hanging nodes} of at most one level.
Let $\calN_{h}$ be the set of nodes of the partition $\Omega_{h}$ 
and let $\calN_{h}^I \cup \calN_{h}^H$ 
be a partition of $\calN_{h}$ 
into unconstrained nodes $\calN_{h}^I$
and hanging nodes $\calN_{h}^H$; see Figure~\ref{fig:patch}.
In addition, we let $\calN_{h}^{I}(K)$ be the set of nodes in $\calN_{h}^{I}$
that are contained in an element $K \in \Omega_{h}$. 
We write $\calF_{h}$ for the set of faces of the quadrilaterals  
in the partition. We let $\calF_{h}^{I}$ denote the subset of $\calF_{h}$
whose members are not contained in any other member of $\calF_{h}$,
and we let $\calF_{h}^{H} := \calF_{h} \setminus \calF_{h}^{I}$. 
% be the complementary set of $\calF_{h}^{I}$ in $\calF_{h}$. 

We let $\Gamma_{D} \subseteq \partial\Omega$ be a subset of the boundary 
that is the union of faces of the partition.
We also let $\Gamma_{N} \subseteq \partial\Omega$
be the essential complement of $\Gamma_{D}$ in $\partial\Omega$,
which is again a union of faces of the partition.
% FIXME: reference about non-conforming meshes.

% FIXME: "nodes" vs "vertices"

\begin{figure}[t]
  \centering
  \subfloat[]{
    \begin{tikzpicture}[]
      \draw[step=1.0, thick] (-0.2,-0.2) grid (4.2,2.2);
      \draw[step=0.5, thick] (1.0,-0.2) grid (3.0,1.0);
      \draw[step=0.5, thick] (3.0,-0.2) grid (4.2,2.2);
      \node at (1.0,0.5) {$\times$};
      \node at (1.5,1.0) {$\times$};
      \node at (2.5,1.0) {$\times$};
      \node at (3.0,1.5) {$\times$};
      \draw[fill] (2.0,1.0) circle[radius=1.5pt];
      \node at (2.2,1.2) {$p$};
      \draw[fill] (3.0,1.0) circle[radius=1.5pt];
      \node at (3.2,1.2) {$q$};
    \end{tikzpicture}}
  \subfloat[]{
    \begin{tikzpicture}[]
      \draw[step=1.0, white] (0.5,-0.2) grid (3.5,2.2);
      \draw[step=1.0, thick] (1.0,1.0) grid (3.0,2.0);
      \draw[step=0.5, thick] (0.99,0.5) grid (3.0,1.0);
      \node at (1.5,1.0) {$\times$};
      \node at (2.5,1.0) {$\times$};
      \draw[fill] (2.0,1.0) circle[radius=1.5pt];
      \node at (2.2,1.2) {$p$};
    \end{tikzpicture}}
  \subfloat[]{
    \begin{tikzpicture}[]
      \draw[step=1.0, white] (1.5,-0.2) grid (4.2,2.2);
      \draw[step=1.0, thick] (2.0,1.0) grid (3.0,2.0);
      \draw[step=0.5, thick] (1.99,0.5) grid (3.0,1.0);
      \draw[step=0.5, thick] (3.0,0.5) grid (3.5,2.0);
      \node at (2.5,1.0) {$\times$};
      \node at (3.0,1.5) {$\times$};
      \draw[fill] (3.0,1.0) circle[radius=1.5pt];
      \node at (3.2,1.2) {$q$};
    \end{tikzpicture}}
  \caption{
    A quadrilateral mesh (a) with hanging nodes ($\times$) and
    reconstructed patches around nodes $p\in\calN_{h}^I$ (b) and
    $q\in\calN_{h}^I$ (c).
  }
  \label{fig:patch}
\end{figure}
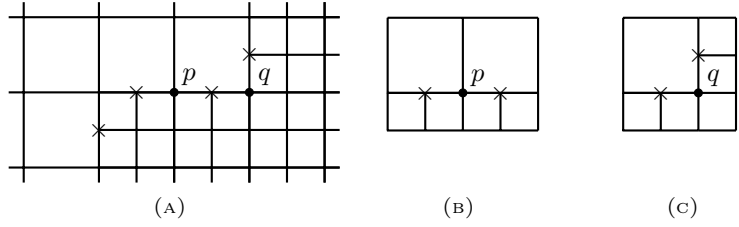

% 
% Describe mesh regularity 
% 
We let $\widehat K$ denote the unit square, 
and for every full-dimensional cell $K \in \Omega_{h}$
we let $\boldsymbol\Psi_{K}\;:\;\widehat K\to K$
be a fixed affine reference transformation. 
We let $C_{h} > 0$ be the \emph{shape-constant} of the partition $\Omega_{h}$,
which is defined as the minimal non-negative number satisfying
$\| \Jacobian\boldsymbol\Psi_{K} \|_{L^{\infty}(\widehat K)} < C_{h}$
and
$\| \Jacobian\boldsymbol\Psi_{K}^{\inv} \|_{L^{\infty}(K)} < C_{h}$
for every full-dimensional quadrilateral $K \in \Omega_{h}$. 
\\

% To ensure shape regularity we assume that for every transformation
% $\boldsymbol\Psi_{K}$ of the unit square $\widehat K$ to a cell $K$,
% $\boldsymbol\Psi_{K}\;:\;\widehat K\to K$, the determinants of the Jacobians of
% the transformation and its inverse are uniformly bounded. 

% In this section we describe the implementation 
% of a patch-based local flux reconstruction 
% that utilizes a partition of unity.

We define several finite element spaces,
beginning with scalar functions.
We write $\bfQ^{r}(K)$ and $\bfQ^{r}(F)$ for the space of polynomials 
over any full-dimensional cell $K \in \Omega_{h}$ 
and any codimension-one face $F \in \calF_{h}$ of the partition, respectively,
with maximum degree $r$ in each variable.
We let $\bfQ^{r}_{-1}(\Omega_{h})$ denote the space of functions 
that are piecewise polynomials of degree at most $r$ in each variable 
with respect to the partition $\Omega_{h}$. 
We have
\begin{gather*}
 \bfQ^{r}_{-1}(\Omega_{h})
 :=
 \sum_{ K \in \Omega_{h} }
 \bfQ^{r}(K)
 . 
\end{gather*}
We define $\bfQ^{r}(\Omega_{h}) := H^{1}(\Omega) \cap \bfQ^{r}_{-1}(\Omega_{h})$
as the $H^{1}$-conforming subspace of $\bfQ^{r}_{-1}(\Omega_{h})$,
and we define $\bfQ^{r}(\Omega_{h},\Gamma_{D}) := H^{1}(\Omega,\Gamma_{D}) \cap \bfQ^{r}_{-1}(\Omega_{h})$
as the subspace thereof whose members have vanishing trace along $\Gamma_{D}$.

% {\matthias
\begin{remark}
  The conformity requirement $\bfQ^{r}(\Omega_{h}) \subset H^{1}(\Omega)$
  is typically enforced by imposing additional interpolatory constraints on
  degrees of freedom associated with hanging nodes (\emph{hanging-node
  constraints}); see \cite{carstensen2009hanging} for a detailed
  discussion.
\end{remark}
% }%

The space $\bfQ^{r}_{-1}(\Omega_{h})$ can be embedded into the space of distributions
over the test space $H^{1}(\Omega,\Gamma_{D})$,
i.e., $\bfQ^{r}_{-1}(\Omega_{h}) \subseteq L^{2}(\Omega) \subseteq H^{-1}(\Omega,\Gamma_{N})$. 
More generally, every $\bfQ^{r}(F)$ for $F \in \calF^{I}_{h}$ with $F \nsubseteq \Gamma_{N}$
can be identified with a distribution over $H^{1}(\Omega,\Gamma_{D})$
by taking the trace onto $F$ and integrating.
Hence we define 
\begin{gather*}
 \bfQ^{r}_{-2}(\Omega_{h},\Gamma_{N})
 :=
 \bfQ^{r}_{-1}(\Omega_{h})
 +  
 \sum_{ \substack{ F \in \calF_{h}^{I} \\ F \nsubseteq \Gamma_{N} } }
 \bfQ^{r}(F)
 ,  
\end{gather*}
which is a distributional finite element subspace of $H^{-1}(\Omega,\Gamma_{N})$.

We also consider finite element spaces of vector fields. 
We let $\RT^{r}(K)$ be the Raviart-Thomas space 
over the full-dimensional cell $K \in \bfQ_{h}$,
which is formally defined as 
\begin{gather*}
 \RT^{r}(K)
 :=
 \bigoplus_{i=1}^{n}
 \left( \bfQ^{r}(K) + x_{i} \bfQ^{r}(K) \right)
 .
\end{gather*}
In other words, the components of each member of $\RT^{r}(K)$
are polynomials of degree at most $r$ in each coordinate variable 
except for the $i$-th variable, which has at most degree $r+1$.

% Taken from Arnold-boffi-bonizzoni
% See Brezzi-Fortin book, p.119, definition (3.20).
Successively, we introduce the broken Raviart-Thomas space $\RT_{-1}^{r}(\Omega_{h})$
of degree $r$ with respect to the partition $\Omega_{h}$, 
(see also \cite{arnold2015finite}), 
which is 
\begin{gather*}
 \RT_{-1}^{r}(\Omega_{h})
 :=
 \sum_{ K \in \Omega_{h} }
 \RT_{-1}^{r}(K)
 . 
\end{gather*}
We define $\RT^{r}(\Omega_{h},\Gamma_{N}) := \bfH(\Omega,\Gamma_{N},\divergence) \cap \RT_{-1}^{r}(\Omega_{h})$. 
We obviously have a well-defined divergence operator 
\begin{gather*}
 \divergence : \RT^{r}(\Omega_{h},\Gamma_{N}) \rightarrow \bfQ^{r}_{-1}(\Omega_{h}).
\end{gather*}
More generally, we introduce the piecewise divergence operator 
\begin{gather*}
 \divergence_{h} : \RT^{r}_{-1}(\Omega_{h}) \rightarrow \bfQ^{r}_{-1}(\Omega_{h}).
\end{gather*}
In addition to that, we consider the jump term operator
\begin{gather*}
 [\cdot]_{\Omega_{h},\Gamma_N} 
 : 
 \RT^{r}_{-1}(\Omega_{h}) 
 \rightarrow 
 \sum_{ \substack{ F \in \calF_{h}^{I} \\ F \nsubseteq \Gamma_{N} } } \bfQ^{r}(F)
 .
\end{gather*}
Finally, the divergence operator on the broken Raviart-Thomas space 
is reintroduced in the sense of distributions:
\begin{gather*}
 \divergence_{\Omega,\Gamma_N} : \RT^{r}_{-1}(\Omega_{h}) \rightarrow \bfQ^{r}_{-2}(\Omega_{h},\Gamma_{N}),
 \quad 
 \tau_{h} 
 \mapsto
 \divergence_{h} \tau_{h} 
 -
 [\tau_{h}]_{\Omega_{h},\Gamma_{N}}
 .
\end{gather*}
A fundamental observation, 
which will be proven shortly, 
is that for every $s_{h} \in \bfQ^{r}_{-2}(\Omega_{h},\Gamma_{N})$
there exists $\varrho_{h} \in \RT^{r}_{-1}(\Omega_{h})$
that solves the flux equation 
\begin{gather*}
 - \divergence_{\Omega,\Gamma_N} \varrho_{h} = s_{h} - \scrH_{\Omega,\Gamma_{D}} s_{h}
 .
\end{gather*}
This shows the purpose of our definition of $\bfQ^{r}_{-2}(\Omega_{h},\Gamma_{N})$
as giving the right target space for the distributional divergence 
on the broken Raviart-Thomas space.
Let us now prove this existence result. 

\begin{lemma}
 \label{lemma:einzigeslemma}
 Let $s_{h} \in \bfQ^{r}_{-2}(\Omega_{h},\Gamma_{N})$.
 Then there exists $\varrho_{h} \in \RT^{r}_{-1}(\Omega_{h})$
 such that 
 \begin{gather*}
  - \divergence_{\Omega,\Gamma_N} \varrho_{h} = s_{h} - \scrH_{\Omega,\Gamma_{D}} s_{h}
  .
 \end{gather*}
\end{lemma}

\begin{proof}
%  Due to scaling arguments and mesh regularity 
%  it is sufficient to consider only a fixed domain $\omega$ with a fixed mesh $\omega_{h}$. 
%  We let $\gamma_{D}$ be a boundary patch of $\omega$
%  that is partitioned by outer faces of $\omega_{h}$.
%  
 We fix $s_{h} \in \bfQ^{r}_{-2}(\Omega_{h},\Gamma_{N})$,
 which by definition can be written in the form 
 \begin{gather*}
  s_{h} = \sum_{ K \in \Omega_{h} } f_{K} + \sum_{ F \in \calF^{I}_{h} } j_{F}
  ,
 \end{gather*}
 where $f_{K} \in \bfQ^{r}(K)$ for each $K \in \Omega_{h}$ 
 and $j_{F} \in \bfQ^{r}(F)$ for each $F \in \calF^{I}_{h}$. 
 
 We first note that there exists $\varrho'_{h} \in \RT^{r}_{-1}(\Omega_{h})$ such that 
 \begin{gather*}
  - [ \varrho'_{h} ]_{\Omega_{h},\Gamma_{N}} = \sum_{ F \in \calF^{I}_{h} } j_{F}.
 \end{gather*}
 We set $s_{h}' := s_{h} + \divergence_{\Omega,\Gamma_{N}} \varrho'_{h}
 = s_{h} + \divergence_{h} \varrho'_{h}$.
 By construction, $s_{h}' \in \bfQ^{r}_{-1}(\Omega_{h})$,
 and $s_{h}'$ annihilates $\scrH(\Omega,\Gamma_{D})$ 
 if and only if 
 $s_{h}$ annihilates $\scrH(\Omega,\Gamma_{D})$. 
 It remains to construct $\tau_{h} \in \RT^{r}(\Omega_{h},\Gamma_{N})$ 
 such that $\divergence_{\Omega,\Gamma_{N}} \tau_{h} = s_{h}'$
 because then the desired vector field $\varrho_{h} \in \RT^{r}_{-1}(\Omega)$
 is given by 
 \begin{gather*}
  \varrho_{h} := \varrho_{h}' + \tau_{h}
  .
 \end{gather*}
 For this purpose we introduce the canonical interpolators
 \begin{gather*}
  \bfI\left[\RT^{r}(\Omega)\right]
  : C^{\infty}(\Omega)^{n} \rightarrow \RT^{r}(\Omega),
  \\
  \bfI\left[\bfQ^{r}(\Omega)\right] 
  : C^{\infty}(\Omega) \rightarrow \bfQ^{r}(\Omega).
 \end{gather*}
 The former interpolator is defined by 
 \begin{gather*}
  \int_{K} \bfI\left[\RT^{r}(\Omega)\right] \tau \cdot \phi_{K}
  =
  \int_{K} \tau \cdot \phi_{K},
  \quad 
  \phi_{K} 
  \in 
  \bigoplus_{i=1}^{n} \partial_{i} \bfQ^{r}(K)
  ,
 \end{gather*}
 for all $K \in \Omega_{h}$
 and by 
 \begin{gather*}
  \int_{F} \trace_{F} \bfI\left[\RT^{r}(\Omega)\right] \tau \cdot \xi_{K}
  =
  \int_{F} \trace_{F} \tau \cdot \xi_{K},
  \quad 
  \xi_{K} \in \bfQ^{r}(F) % OK!
 \end{gather*}
 for $F \in \calF^{I}_{h}$.
 The latter interpolator is defined by 
 \begin{gather*}
  \int_{K} \bfI\left[\bfQ^{r}(\Omega)\right] f \cdot g_{K}
  =
  \int_{K} f \cdot g_{K},
  \quad 
  K \in \Omega_{h},
  \quad 
  g_{K} \in \bfQ^{r}(K)
 \end{gather*}
 for $f \in C^{\infty}(\Omega)$. 
 Indeed, it follows from the discussion in Section~5 of \cite{arnold2015finite}
 that this defines members of the finite element spaces
 $\RT^{r}(\Omega_{h})$ and $\bfQ^{r}_{-1}(\Omega_{h})$, respectively.
 
 For every $\tau \in C^{\infty}(\Omega)^{n}$, 
 $K \in \Omega_{h}$, and $g_{K}\in\bfQ^{r}(K)$
 we find 
 \begin{align*}
  \int_{K} \bfI\left[\bfQ^{r}(\Omega)\right] \divergence \tau \cdot g_{K}
  &=
  \int_{K} \divergence \tau \cdot g_{K}
  \\&=
  \int_{K} \tau \cdot \nabla g_{K}
  +
  \int_{\partial K} \trace_{N} \tau \cdot \trace_{} g_{K}
  \\&=
  \int_{K} \bfI\left[\RT^{r}(\Omega)\right] \tau \cdot \nabla g_{K}
  +
  \int_{\partial K} \trace_{N} \tau \cdot \trace_{} g_{K}
  . 
 \end{align*}
 The boundary integral is given as a combination of face integrals
 over $K$. On a uniform mesh, we could now replace the face integrals
 of $\tau$ by the face integrals of $\bfI\left[\RT^{r}(\Omega)\right] \tau$
 and revert the integration by parts. But since we allow meshes 
 with hanging nodes, it may happen that the faces of $K$ do not represent 
 degrees of freedom over $\RT^{r}(\Omega)$.
 The trick is that a repeated application of the integration 
 by parts formula allows us to express the integral as a combination 
 of degrees of freedom of $\tau$ associated to further full-dimensional 
 quadrilaterals and associated to faces in $\calF^{I}_{h}$. 
 We can then replace these by degrees of freedom of 
 $\bfI\left[\RT^{r}(\Omega)\right] \tau$ and apply the integration by parts
 formulas in the reverse order to obtain that 
 \begin{align*}
  \int_{K} \bfI\left[\bfQ^{r}(\Omega)\right] \divergence \tau \cdot g_{K}
  =
  \int_{K} \divergence \bfI\left[\RT^{r}(\Omega)\right] \tau \cdot g_{K}
  . 
 \end{align*}
 This eventually implies that 
 \begin{gather*}
  \divergence \bfI\left[\RT^{r}(\Omega)\right] \tau
  =
  \bfI\left[\bfQ^{r}(\Omega)\right] \divergence \tau
  .
 \end{gather*}
 So the canonical interpolants commute with the exterior derivative. 
 
 Since functions in $\bfH(\Omega,\divergence)$
 have well-defined traces on the codimension one faces of $\Omega_{h}$
 with regularity $\bfH^{-\onehalf}$,
 we conclude that we have bounded operators 
 \begin{gather*}
  \bfI\left[\RT^{r}(\Omega)\right]
  : \bfH(\Omega,\divergence) \rightarrow \RT^{r}(\Omega) \subset \bfL^{2}(\Omega),
  \\
  \bfI\left[\bfQ^{r}(\Omega)\right] 
  : L^{2}(\Omega) \rightarrow \bfQ^{r}(\Omega) \subset L^{2}(\Omega).
 \end{gather*}
 Moreover, when $\tau \in \bfH(\Omega,\divergence)$
 has vanishing normal trace on the faces of $\Gamma_{N}$,
 then the same holds true for its canonical interpolation 
 $\bfI\left[\RT^{r}(\Omega)\right] (\tau)$. 
 In particular, we have a bounded operator 
 \begin{gather*}
  \bfI\left[\RT^{r}(\Omega)\right]
  : \bfH(\Omega,\divergence,\Gamma_{N}) \rightarrow \RT^{r}(\Omega_{h},\Gamma_{N}) \subset \bfL^{2}(\Omega).
 \end{gather*}
 These observations are applied as follows.
 There exists $\tau \in \bfH(\Omega,\divergence,\Gamma_{N})$
 such that $\divergence \tau = s_{h}'$, 
 as follows from the Poincar\'e-Friedrichs inequality \eqref{math:poincarefriedrichs:distdiv}
 and the fact that $s_{h}' \in L^{2}(\Omega)$ annihilates $\scrH(\Omega,\Gamma_{D})$. 
 We set $\tau_{h} = \bfI\left[\RT^{r}(\Omega)\right] \tau$ 
 and see 
 \begin{align*}
  \divergence \tau_{h}
  &=
  \divergence \bfI\left[\RT^{r}(\Omega)\right] \tau 
  \\&=
  \bfI\left[\bfQ^{r}(\Omega)\right] \divergence \tau 
  =
  \bfI\left[\bfQ^{r}(\Omega)\right] s_{h}' 
  =
  s_{h}' 
  . 
 \end{align*}
 This completes the proof.
\end{proof}

\subsection{Localized Flux Reconstruction}
\label{subsec:flux:localized}

After this preparation, we approach the solution of the residual flux equation. 
Given a distributional finite element right-hand side $F_{h} \in \bfQ^{r}_{-2}(\Omega_{h},\Gamma_{N})$, 
we let $u_{h} \in \bfQ^{r+1}(\Omega_{h},\Gamma_{D})$ 
be the unique solution of 
\begin{gather*}
 \int_{\Omega} \grad u_{h} \cdot A \grad v_{h}
 =
 F_{h} \left( v_{h} - \scrH_{\Omega,\Gamma_{D}} v_{h} \right)
 ,
 \quad 
 v_{h} \in \bfQ^{r+1}(\Omega_{h},\Gamma_{D})
 .
\end{gather*}
In order to facilitate the equilibrated error estimator 
we seek a solution $\sigma_{h} \in \RT^{r}_{-1}(\Omega_{h})$
to the finite element flux equation 
\begin{gather}
 \label{eq:finiteelementfluxequation}
 - \divergence_{\Omega,\Gamma_{N}} A \sigma_{h} = F_{h} - \scrH_{\Omega,\Gamma_{D}} F_{h}
 .
\end{gather}
Note that \eqref{eq:finiteelementfluxequation} is well-posed
according to Lemma~\ref{lemma:einzigeslemma}.
% since the right-hand side annihilates $\scrH(\Omega,\Gamma_{D})$.
Hence, this equation can be solved in the sense of least-squares with respect to the
$A$-norm. Moreover, if $F_{h} \in \bfQ^{r}_{-1}(\Omega)$ is a
square-integrable function, then the variable $\sigma_{h}$ may be sought in
$\RT^{r}(\Omega_{h},\Gamma_{N})$. In this sense \eqref{eq:finiteelementfluxequation}, 
the finite element flux equation can be solved with computational costs
comparable to a mixed finite element method. In particular, the approximate
solution $u_{h}$ does not enter the construction.
\\

However, the Galerkin approximation $u_{h}$ contains additional information
that facilitates a localized solution of the flux equation. The solution
constructed in that manner will generally not minimize the $A$-norm,
though.

To begin with, we define the residual 
$r_{h} \in \bfQ^{r}_{-2}(\Omega_{h},\Gamma_{N}) \subseteq H^{-1}(\Omega,\Gamma_{N})$ by
\begin{gather*}
  \label{eq:residualdefinition}
  r_{h}
  :=
  \left( \Id - \scrH_{\Omega,\Gamma_{D}} \right) F_{h} + \divergence_{\Omega,\Gamma_{N}} A \grad u_{h}
  .
\end{gather*}
Our goal is 
to construct a vector field $\varrho_{h} \in \RT^{r}_{-1}(\Omega_{h})$ solving
the finite element residual flux equation 
\begin{gather}
  \label{eq:FE:fluxequation:residual}
  - \divergence_{\Omega,\Gamma_{N}} A \varrho_{h} = r_{h}.
\end{gather}
Again, this system has a proper solution, which in theory can be directly
constructed from the any solution of \eqref{eq:finiteelementfluxequation}.
% {\matthias
For the implementation of the localized flux reconstruction, though, 
we construct a solution of \eqref{eq:finiteelementfluxequation} 
with the help of a solution of \eqref{eq:FE:fluxequation:residual}. 
% }
\\

In order to discuss this, we introduce a partition of unity. 
For every non-hanging node $v \in \calN_{h}^I$ 
we let $\psi^{V} \in \bfQ^{1}(\Omega_{h})$
be defined by requiring that $\psi^{V}$ 
takes the value $1$ at the node $v \in \calN_{h}^{I}$ 
and the value $0$ at all other non-hanging nodes in $\calN_{h}^{I}$.
The support $\omega^{V} := \supp \psi^{V}$ is a subdomain of $\Omega$. 
It is evident that the collection of all such $\psi^{V}$
constitutes a partition of unity over $\Omega$. 

We let $\omega^{V}_{h}$ be the \emph{patch} of elements around $v$, 
which forms a partition of $\omega^{V}$.
In addition, we define the local boundary patches
\begin{gather*}
 \gamma_{D}^{V} 
 :=
 \left\{\begin{array}{rl}
         \partial\omega^{V} \cap \Gamma_{D} & \text{ if } V    \in \Gamma_{D},
         \\
         \emptyset                          & \text{ if } V \notin \Gamma_{D},
        \end{array}\right.
 \qquad 
 \gamma_{N}^{V} := \partial\omega^{V} \setminus \overline{\gamma_{D}^{V}}.
\end{gather*}
It is important to note that every local patch $\omega^{V}$
is a contractible Lipschitz domain. 
In particular, $\scrH(\omega^{V},\gamma^{V}_{D})$ 
is either zero 
or spanned by the constant functions over the patch $\omega^{V}$
depending on whether $\gamma^{V}_{D}$ is non-empty or not. 

For every $v \in \calN_{h}^{I}$ being a non-hanging node,
we define the localized residual as the distribution 
$r_{h}^{V} := \psi^{V} r_{h}$.
Note that $r_{h}^{V} \in \bfQ^{r+1}_{-2}(\omega_{h},\gamma_{N}^{V})$.
Moreover,
for any function $1_{V} \in \scrH_{D}(\omega^{V},\gamma^{V}_{D})$
with constant value $1$ we have 
\begin{gather*}
 r^{V}_{h}( 1_{V} ) = r^{V}_{h}( \psi_{h}^{V} ) = 0.
\end{gather*}
This follows by the definition of the residual together with Galerkin
orthogonality.

Consequently, Lemma~\ref{lemma:einzigeslemma} gives the existence of a solution 
$\varrho_{h}^{V} \in \RT_{-1}^{r+1}(\omega^{V})$ 
to the localized problem 
\begin{gather*}
 - \divergence_{\Omega,\Gamma_{N}} \varrho_{h}^{V} = r^{V}_{h}
 .
\end{gather*}
Definitions show that the sum  
\begin{align}
  \varrho_{h}^{L} := \sum_{ v \in \calN_{h}^I }\varrho_{h}^{V}
\end{align}
is a solution to the residual flux equation \eqref{eq:FE:fluxequation:residual}.
A solution to the finite element flux equation
is then given by 
\begin{gather*}
 \sigma_{h}^{L} := \varrho_{h} + \grad u_{h}
 \in
 \RT^{r+1}\left( \Omega_{h}, \Gamma_{N} \right)
 .
\end{gather*}

% TODO: Discuss the implementation in Deal.ii based on the original notes by Matthias.
\subsection{Applications}
\label{subsec:flux:applications}

We assume that $F \in H^{-1}(\Omega,\Gamma_N)$ 
and that $u \in H^{1}(\Omega,\Gamma_{D})$ is the unique solution of 
\begin{gather}
 \label{math:approximatepoisson}
 - \divergence_{\Omega,\Gamma_N} A \grad u = F - \scrH_{\Omega,\Gamma_{D}} F,
 \quad 
 u \perp \scrH(\Omega,\Gamma_{D})
 .
\end{gather}
Under the special assumption that $F \in \bfQ^{r}(\Omega_{h},\Gamma_{N})$
we can immediately instantiate the flux reconstruction (either global or localized)
to obtain 
\begin{gather*}
 \| \grad u - \grad u_{h} \|_{A} \leq \| \grad u_{h} - \sigma_{h} \|_{A}
 .
\end{gather*}
In general, we cannot assume that $F$ is a member of
$\bfQ^{r}_{-2}(\Omega_{h},\Gamma_{N})$. 
Hence the flux reconstruction is performed 
by solving an approximate flux equation 
with an approximate right-hand side $F_{h} \in \bfQ^{r}_{-2}(\Omega,\Gamma_{N})$. 
Given any discrete flux $\sigma_{h} \in
\RT^{r+1}(\Omega_{h},\Gamma_{N})$ solving
\begin{gather}
 \label{math:approximatefluxequation}
 - \divergence_{\Omega,\Gamma_N} A \sigma_{h} = F_{h} - \scrH_{\Omega,\Gamma_{D}} F_{h},
\end{gather}
we can utilize the generalized error estimates of the previous
section.

The discrete flux $\sigma_{h} \in \RT^{r+1}(\Omega_{h},\Gamma_{N})$
can of course be obtained by solving a global finite element system.
% {\matthias
Let $\sigma_{h}^{G} \in \RT^{r+1}(\Omega_{h},\Gamma_{N})$ denote 
this flux reconstruction, which is obtained by solving 
\eqref{math:approximatefluxequation} 
with minimal $A$-norm over $\Omega$.
% }%
We also write $\varrho_{h}^{G} := \sigma_{h}^{G} - \nabla u_{h} \in
\RT^{r+1}_{-1}(\Omega_{h})$ for the corresponding solution of the residual
flux equation.

% {\matthias%
The localization
% }
of the flux reconstruction with approximate data requires additional assumptions. 
The reason is that the localized flux reconstruction requires
a local equilibration condition; since $u_{h}$ satisfies the Galerkin
condition with respect to $F$ but generally not with respect to $F_{h}$, it
is not immediately clear how the localized construction can be generalized.

As a solution we impose an additional condition: we require that
\begin{gather}
 \label{math:dataregularity}
 \left( F - F_{h} \right)( v_{h} ) = 0,
 \quad 
 v_{h} \in \bfQ^{1}(\Omega_{h},\Gamma_{D})
 .
\end{gather}
Under that condition we have 
\begin{gather*}
 F_{h}\left( \psi^{V} \right) - \left\langle \grad u_{h}, \grad \psi^{V} \right\rangle_{A}
 =
 F    \left( \psi^{V} \right) - \left\langle \grad u_{h}, \grad \psi^{V} \right\rangle_{A}
\end{gather*}
for all non-hanging nodes $V \in \calN^{I}_{h}$ with $V \notin \Gamma_{D}$.
Consequently, the local problems in the localized flux reconstruction 
are well-posed,
and thus a solution $\varrho_{h} \in \RT^{r+1}_{-1}(\Omega_{h})$
to the \emph{approximate residual flux equation} 
\begin{gather*}
 - \divergence_{\Omega,\Gamma_N} A \varrho_{h}
 =
 \left( \Id - \scrH_{\Omega,\Gamma_{D}} \right) F_{h} + \divergence A \grad u_{h} 
\end{gather*}
is found by adding the local solutions.
As before, a solution to \eqref{math:approximatefluxequation} is 
now given by $\sigma_{h} := \varrho_{h} - \grad u_{h}$.
A reliable error estimates (even in the case $F \neq F_{h}$)
is then given via \eqref{math:reliableerrorestimate}. %, or \eqref{math:RSS:estimate}. 
To quantify the efficiency of this error estimate 
we may use \eqref{math:efficiencyofestimate}. % or \eqref{math:RRS:efficiency}.

Finally, 
for our study of adaptive finite element methods later in this article 
we introduce the local error indicators 
\begin{gather}
 \label{math:localerrorindicator}
 \eta_{K}(\varrho_{h})
 :=
 \| \varrho_{h} \|_{\bfL^{2}_{A}(K)}^{2}
 . 
\end{gather}

\subsection{Improved Estimates}
\label{subsec:flux:improvedestimates}
 The aforementioned error estimates and efficiency estimates are of very
 general nature. The additional structure given by the finite element
 setting enables some interesting further results that quantify the
 influence of the error $F - F_{h}$ in the data. These techniques, however,
 utilize additional regularity assumptions:
 \begin{gather}
  \label{math:l2regularityassumption}
  F, F_{h} \in L^{2}(\Omega)
  ,
  \\
  \label{math:orthogonalityassumption}
  F - F_{h} \perp \bfQ_{-1}^{0}(\Omega_{h})
  .
 \end{gather}
 We begin with a lower bound for the error estimator. 
% %  We find 
% %  \begin{align*}
% %   \| \nabla u - \nabla u_{h} \|_{A}
% %   &\leq 
% %   \| \nabla \widehat u - \nabla u_{h} \|_{A}
% %   +
% %   \| \nabla u - \nabla \widehat u \|_{A}
% %   \\&\leq 
% %   \| \varrho_{h} \|_{A}
% %   +
% %   \| \nabla u - \nabla \widehat u \|_{A}
% %   . 
% %  \end{align*}
% %  The orthogonality condition \eqref{math:orthogonalityassumption}
% %  allows for a piecewise Poincar\'e-Friedrichs inequality
% %  which gives 
% %  \begin{align*}
% %   \| \nabla u - \nabla \widehat u \|_{A}
% %   &\leq 
% %   \sum_{ K \in \Omega_{h} }
% %   \| \nabla u - \nabla \widehat u \|_{\bfL^{2}_{A}(K)}
% % %   \\&
% %   \leq 
% %   \sum_{ K \in \Omega_{h} }
% %   h_{K}
% %   \| F - F_{h} \|_{L^{2}(K)}
% %   .
% %  \end{align*}
% %  Hence we obtain 
% %  \begin{gather}
% %   \label{math:improvedestimates:reliability}
% %   \| \nabla u - \nabla u_{h} \|_{A}
% %   \leq 
% %   \| \varrho_{h} \|_{A}
% %   +
% %   \sum_{ K \in \Omega_{h} } h_{K} \| F - F_{h} \|_{ L^{2}(K) }
% %   . 
% %  \end{gather}
 It is a basic fact that 
 \begin{align*}
  \| \nabla u - \nabla u_{h} \|_{A}
  &=
  \sup_{ \substack{ w \in H^{1}(\Omega,\Gamma_{D}), \\ w \perp \scrH(\Omega,\Gamma_{D}) } }
  \dfrac{ 
   \langle \nabla u - \nabla u_{h}, \nabla w \rangle_{A}
  }{
   \| \nabla w \|_{A}
  }
  . 
 \end{align*}
 Let $\widehat u$ be the solution of the Poisson equation 
 analogous to \eqref{math:approximatepoisson}
 but with the data $F$ being replaced by the approximate data $F_{h}$.
 If $\varrho_{h} \in \RT_{-1}^{r+1}(\Omega_{h})$ solves
 \eqref{math:approximatefluxequation}, then we can use 
 \begin{align*}
  \langle \nabla u - \nabla u_{h}, \nabla w \rangle_{A}
  &=
  \langle \nabla u - \nabla \widehat u, \nabla w \rangle_{A}
  +
  \langle \nabla \widehat u - \nabla u_{h}, \nabla w \rangle_{A}
  \\&=
  \langle F - F_{h}, w \rangle 
  +
  \langle \nabla \widehat u - \nabla u_{h}, \nabla w \rangle_{A}
  \\&=
  \langle F - F_{h}, w \rangle
  +
  \langle \varrho_{h}, \nabla w \rangle_{A}
  \\&=
  \langle F - F_{h}, w - w_{h} \rangle
  +
  \langle \varrho_{h}, \nabla w \rangle_{A}
  , 
 \end{align*}
 where in the last step $w_{h} \in \bfQ^{1}(\Omega,\Gamma_{D})$ can be
 chosen arbitrarily due to condition \eqref{math:dataregularity}. By
 picking $w_{h}$ as the first-order Cl\'ement interpolant of $w$ (see
 \cite{carstensen2009hanging} for a discussion) one now sees that 
 \begin{gather}
  \label{math:improvedestimates:reliability}
  \| \nabla u - \nabla u_{h} \|_{A}
  \leq 
  \| \varrho_{h} \|_{A}
  +
  C_{I} \sum_{ K \in \Omega_{h} } h_{K} \| F - F_{h} \|_{ L^{2}(K) }
  , 
 \end{gather}
 where the constant $C_{I} > 0$ depends only on $C_{h}$. 
 
 On the other hand, we prove a converse upper bound for the error estimator
 specifically for the residual flux reconstruction $\varrho_{h}^{L}$. 
 We have 
 \begin{gather*}
  \left\| \varrho_{h}^{L} \right\|_{\bfL^{2}_{A}(K)}
  \leq
  \sum_{ V \in \calN_{h}^{I}(K) }
  \left\| \varrho_{h}^{V} \right\|_{\bfL^{2}_{A}(K)}
  \leq
  \sum_{ V \in \calN_{h}^{I}(K) }
  \left\| \varrho_{h}^{V} \right\|_{\bfL^{2}_{A}(\omega^{V})}
  . 
 \end{gather*}
 A scaling argument with a constant $C>0$
 depending only on the mesh-constant and the polynomial degree 
 yields that 
 \begin{gather*}
  \left\| \varrho_{h}^{V} \right\|_{\bfL^{2}_{A}(\omega^{V})}
  \leq
  C
  \| \nabla u - \nabla u_{h} \|_{\bfL^{2}_{A}(\omega^{V})}
  +
  C
  \sum_{ K \in \omega^{V}_{h} }
  h_{K} \| F_{h} - F \|_{L^{2}(K)}
  . 
 \end{gather*}
 % TODO: Still not clear how easy that is. Can we do better?
%  We now apply the triangle inequality
%  and use a local Poincar\'e-Friedrichs inequality,
%  owing to the orthogonality of $F - F_{h}$ to the piecewise constant functions:
%  \begin{align*}
%   \| \nabla \widehat u - \nabla u_{h} \|_{\bfL^{2}_{A}(K)}
%   &\leq 
%   \| \nabla u - \nabla u_{h} \|_{\bfL^{2}_{A}(K)}
%   +
%   \| \nabla u - \nabla \widehat u \|_{\bfL^{2}_{A}(K)}
%   \\&\leq 
%   \| \nabla u - \nabla u_{h} \|_{\bfL^{2}_{A}(K)}
%   +
%   h_K \| F - F_{h} \|_{L^{2}(K)}
%  \end{align*}
%  for every cell $K \in \Omega_{h}$. 
%  Consequently
%  Using arguments that can be found in the monograph of Braess \cite{braess2007finite}
%  we find 
%  \begin{gather}
%   \label{math:improvedestimates:efficiency}
%   \left\| \varrho_{h}^{L} \right\|_{\bfL^{2}_{A}(K)}
%   \leq 
%   C \sum_{ K \in \Omega_{h} }
%   \left( 
%    \| \nabla u - \nabla u_{h} \|_{\bfL^{2}_{A}(\omega^{V})}
%    + 
%    h_K \| F - F_{h} \|_{L^{2}(\omega^{V})}
%   \right)
%   .
%  \end{gather}
%  Here, $C > 0$ depends only on the mesh-constant and the polynomial degree. 
%  \\
 
 Finally, we address the relation between the flux reconstructions
 $\sigma_{h}^{L}$, the minimum norm solution $\sigma_{h}^{G}$ over
 $\RT^{r+1}(\Omega_{h},\Gamma_{N})$, and the true flux $\nabla u$.
 Since the flux reconstruction $\sigma_{h}^{L}$ can be obtained just as
 much from the distributional divergence of the flux reconstruction $\sigma_{h}^{G}$, a
 scaling argument easily gives 
 \begin{gather}
  \label{math:flux:globalvslocal}
  \left\| \varrho_{h}^{L} \right\|_{\bfL^{2}_{A}(K)}
  +
  \left\| \varrho_{h}^{G} - \varrho_{h}^{L} \right\|_{\bfL^{2}_{A}(K)}
  \leq 
  C
  \sum_{ V \in \calN_{h}^{I}(K) }
  \left\| \varrho_{h}^{G} \right\|_{\bfL^{2}_{A}(\omega^{V})}
  . 
 \end{gather}
 We conclude that the local flux reconstruction will typically 
 not be much worse than the global flux reconstruction in numerical tests. 
 Note that the generic constant $C > 0$ in \eqref{math:flux:globalvslocal}, 
 which depends only the mesh constant and the polynomial, 
 satisfies a computable bound.
 
 The global flux reconstruction is equivalent to the solution 
 of a mixed finite element method. 
 In typical applications $\nabla u$ and $\nabla \widehat u$
 are contained in a higher order Sobolev space such as $\bfH^{s}(\Omega)$
 for a parameter $s \in [0.5,1]$ that depends only on the domain. 
 Consequently, we have 
 \begin{align*}
  \left\| \sigma_{h}^{G} - \nabla u \right\|_{A}
  &\leq 
  \left\| \nabla \widehat u - \nabla u \right\|_{A}
  +
  \left\| \sigma_{h}^{G} - \nabla \widehat u \right\|_{A}
  \\&\leq 
  C 
  \sum_{K \in \Omega_{h}}
  \left( 
   h_{K} \left\| F_{h} - F \right\|_{A}
   +
   h^{s}_{K} \left\| \nabla \widehat u \right\|_{\bfH^{s}(K)}
  \right)
  .
 \end{align*}
 Hence the differences 
 \begin{gather*}
  \theta^{G}_{h} := \sigma^{G}_{h} - \nabla \widehat u,
  \quad 
  \theta^{L}_{h} := \sigma^{L}_{h} - \nabla \widehat u,
 \end{gather*}
 which are members of $\bfX(\Omega,\Gamma_{N},A)$
 as defined in \eqref{math:hodgecomponent}, 
 are controlled by terms of higher order.
 The two vector fields $\theta^{G}_{h}$ and $\theta^{L}_{h}$
 measure how far the respective flux reconstructions 
 differ from the true gradients.

\begin{remark}
 Our analysis of the influence of $F - F_{h}$
 on the overall data is inspired by earlier discussions in the literature,
 in particular the work of Braess and Sch\"oberl \cite{braess2009equilibrated}.
 Their construction of the local problems in the case $F \neq F_{h}$
 is different from ours though (see \cite{braess2007finite}).
%  We furthermore note that we have derived the lower estimate 
%  \eqref{math:improvedestimates:efficiency}
%  without resorting to the classical residual error estimator, 
%  unlike in \cite{braess2007finite}.
\end{remark}

\begin{remark}
 The label \emph{equilibrated error estimator} is motivated by the localized solution of the flux equation
 based on local Neumann problems whose right-hand sides satisfy the equilibrium condition,
 i.e., annihilate the constant functions. 
 In this article we treat error estimators based on solving local Poisson problems in mixed formulations.
 The term \emph{equilibrated error estimator}, however, 
 is shared with another family of error estimators that solve local Poisson problems in elliptic formulation 
 over either local patches or single elements 
 (see \cite{kelly1984self,bank1985some,ainsworth1993unified,morin2003local,ainsworth2007analysis}).
 Even though there exists extensive literature on the latter
 type of error estimators for a variety of different finite element spaces,
 it seems that much less has been published for the former family of
 equilibrated error estimators.

 Equilibrated error estimators are often contrasted to the classical
 residual error estimator: then the latter is called \emph{explicit}
 because it derives estimates directly from bounding the negative norm of
 the residual, which involves mesh-dependent norms on the volume and face
 terms of the residual. Equilibrated residual error estimators, on the
 other hand, are called \emph{implicit} because they bound the negative
 norm of the residual 
 in terms of a flux reconstruction 
 associated with the residual.
\end{remark}

\section{Error Estimation for Quantities of Interest}
\label{sec:hypercircle:qoi}

In many applications we are primarily interested in determining a specific
\emph{quantity of interest}. For the sake of simplicity, we assume 
these to be linear functionals of the solution. Goal-oriented a posteriori
error estimation aims at sharp error estimates in and adaptivity optimized
towards the approximation of the quantity of interest. All but the most
simplistic approaches towards goal-oriented a posteriori error estimation 
require the computation of Galerkin solutions of both the original primal and a dual
problem of the same kind.
In this section we review some techniques proposed 
in the literature \cite{becker2001optimal,mozolevski2015goal,feischl2016abstract}
and suggest a new heuristic approximation for the error 
in the quantity of interest.

\subsection{Basic Theory}
\label{subsec:hypercircle:qoi:basic}

We reconsider the Poisson problem from Section~\ref{sec:hypercircle:poisson},
where $A \in L^{\infty}(\Omega)^{n\times n}$ is an admissible metric tensor 
and $F \in H^{\inv}(\Omega,\Gamma_{N})$ is the right-hand side.
Additionally, we let $J \in H^{\inv}(\Omega,\Gamma_{N})$ be a functional.
We are interest in the value $J(u)$, 
where $u \in H^{1}(\Omega,\Gamma_{D})$ solves
the following Poisson problem: 
% \begin{subequations}
\begin{gather}
 \label{math:poisson:primal}
 - \divergence_{\Omega,\Gamma_{N}} A \grad u = F - \scrH_{\Omega,\Gamma_{D}} F,
 \qquad
 u \perp \scrH(\Omega,\Gamma_{D}).
\end{gather}
% \end{subequations}
The functional $J$ is also called \emph{goal functional} in the literature.
Now suppose that $u_{h} \in H^{1}(\Omega,\Gamma_{D})$ is any computable approximation of $u$.
We are interested in estimating the error $J(u) - J(u_{h})$ in the goal functional
without using the generally unknown true solution of the Poisson problem.
\\

A central concept in the error estimation of linear functionals is the \emph{dual problem}.
This constitutes in determining the \emph{dual quantity of interest} $F(z)$, 
where $z \in H^{1}(\Omega,\Gamma_{D})$ solves the Poisson problem 
% \begin{subequations}
\begin{gather}
 \label{math:poisson:dual}
 - \divergence_{\Omega,\Gamma_{N}} A \grad z = J - \scrH_{\Omega,\Gamma_{D}} J,
 \qquad
 z \perp \scrH(\Omega,\Gamma_{D})
 . 
\end{gather}
% \end{subequations}
Accordingly, we call the problem of determining the \emph{primal quantity of interest} $J(u)$
the \emph{primal problem}.
The well-posedness of \eqref{math:poisson:primal} and
\eqref{math:poisson:dual} is clear by our discussion in Section~\ref{sec:background}.
$F$ is the \emph{primal right-hand side} and the \emph{dual goal functional},
while $J$ is the \emph{dual right-hand side} and the \emph{primal goal functional}. 
We call $u$ and $z$ the \emph{primal solution} and the \emph{dual solution}, respectively.
The elementary identity 
\begin{gather}
 \label{math:qoi:identity}
 J(u) = \langle \grad u, \grad z \rangle_{A} = F(z)
\end{gather}
states that the primal and the dual quantities of interest coincide. 
% 
% In order to improve upon the error estimate \eqref{math:naive:qoi:estimate},
% we combine error estimates for the primal and the dual problem. 
\\

In the sequel we apply this idea to finite element approximations. 
We fix a cubical partition $\Omega_{h}$ of the domain
$\Omega$ such that a subset of the faces is a partition of $\Gamma_{D}$.
Moreover we let $r \in \bbN$ be a fixed non-negative integer that
denotes the polynomial degree. 
We assume to be given 
% Our goal-oriented error estimates require 
both an approximation $u_{h} \in H^{1}(\Omega,\Gamma_{D})$ of the primal solution 
and an approximation $z_{h} \in H^{1}(\Omega,\Gamma_{D})$ of the dual solution.
% TODO: right hand side approximation 
We specifically require these to be Galerkin solutions 
in the sense that 
% in the following sense:
% We assume that $\bfQ^{r}(\Omega,\Gamma_{D}) \subseteq H^{1}(\Omega,\Gamma_{D})$ is a closed subspace of $H^{1}(\Omega,\Gamma_{D})$,
% called the \emph{Galerkin space} in the sequel, 
% and let $u_{h}, z_{h} \in \bfQ^{r}(\Omega,\Gamma_{D})$ satisfy the Galerkin conditions 
$u_{h}, z_{h} \in \bfQ^{r}(\Omega,\Gamma_{D})$ satisfy 
\begin{subequations}
\label{math:galerkinform}
\begin{gather}
 \label{math:galerkinform:primal}
 \langle \grad u_{h}, \grad v_{h} \rangle_{A}
 =
 F\left( v_{h} - \scrH_{\Omega,\Gamma_{D}} v_{h} \right),
 \quad 
 v_{h} \in \bfQ^{r}(\Omega,\Gamma_{D})
 ,
 \\
 \label{math:galerkinform:dual}
 \langle \grad z_{h}, \grad v_{h} \rangle_{A}
 =
 J\left( v_{h} - \scrH_{\Omega,\Gamma_{D}} v_{h} \right),
 \quad 
 v_{h} \in \bfQ^{r}(\Omega,\Gamma_{D})
 .
\end{gather}
\end{subequations}
To ensure uniqueness we may enforce the usual orthogonality conditions 
$u_{h} \perp \scrH(\Omega,\Gamma_{D})$ and $z_{h} \perp \scrH(\Omega,\Gamma_{D})$, 
but this will not be central in the sequel.
% \begin{gather*}
%  u_{h} \perp \scrH(\Omega,\Gamma_{D}),
%  \quad 
%  z_{h} \perp \scrH(\Omega,\Gamma_{D}).
% \end{gather*}
A simple but important consequence of \eqref{math:galerkinform}
is Galerkin orthogonality:
\begin{subequations}
\label{math:galerkinorthogonality}
\begin{gather}
 \label{math:galerkinorthogonality:primal}
 \langle \grad( u - u_{h} ), \grad v_{h} \rangle_{A} = 0,
 \quad 
 v_{h} \in \bfQ^{r}(\Omega,\Gamma_{D}),
 \\
 \label{math:galerkinorthogonality:dual}
 \langle \grad( z - z_{h} ), \grad v_{h} \rangle_{A} = 0,
 \quad 
 v_{h} \in \bfQ^{r}(\Omega,\Gamma_{D}).
\end{gather}
\end{subequations}
We call $u_{h}$ and $z_{h}$ the \emph{primal} and the \emph{dual Galerkin solution}, respectively.
Using definitions and Galerkin orthogonality, we find 
\begin{align}
 \label{math:qoi:erroridentity}
 \begin{split}
 J( u - u_{h} ) 
 =
%  \langle \grad z, \grad u - \grad u_{h} \rangle_{A}
%  \\&=
 \langle \grad u - \grad u_{h}, \grad z - \grad z_{h} \rangle_{A}
%  =
%  \langle \grad z - \grad z_{h}, \grad u \rangle_{A}
 = 
 F( z - z_{h} )
 .
 \end{split}
\end{align}
In other words, the error in the quantity of interest 
equals the product of the primal and the dual gradient errors. 
We derive various goal-oriented error estimators 
by computing an upper bound for the error product 
$\langle \grad u - \grad u_{h}, \grad z - \grad z_{h} \rangle_{A}$,
which we henceforth call the \emph{error of interest}. 

\begin{remark}
 \label{rem:qoi:energyerror}
 The energy error is a special quantity of interest.
 If $J = F$ and $u_h = z_h$, 
 then the primal and the dual quantities of interest \eqref{math:qoi:identity} 
 coincide with the energy norm of the error,
 and the error of interest \eqref{math:qoi:erroridentity} 
 is then the energy error of the Galerkin solution. 
\end{remark}

% \begin{remark}
%  The proposed goal-oriented error estimates will require 
%  not only a Galerkin solution for the primal problem,
%  but also for the dual problem.
%  Since both problems have similar structure, 
%  efficient solution methods for these problems are desirable.
%  We refer to sparse direct methods \cite{duff1986direct}
%  or Krylov subspace methods for multiple right-hand sides \cite{o1980block}.
%  Additionally, error estimates based on hypercircle identities 
%  need competitive algorithms for flux reconstructions. 
%  \mm{I would remove this remark. The latter only works if the problem in
%  question is symmetric.}
% \end{remark}

\begin{remark}
 An additional source of error has not been discussed yet:
 in applications, the Galerkin solutions are 
 not computed exactly but 
 only up to a certain numerical accuracy. 
%  This is due to demands in accuracy and to principle barriers of numerical accuracy. 
 Consequently, 
 a posteriori error estimates must in principle take into account 
 the deviation from Galerkin orthogonality (see \cite{arioli2013interplay}). 
 % \cite{miedlar2011inexact}
 % 
%  Galerkin orthogonality is not satisfied exactly,
%  which must be taken into account in the construction of a posteriori error estimates (see \cite{miedlar2011inexact}).
 This is not within the scope of this article. 
\end{remark}

\subsection{Review of Flux Reconstruction}
\label{subsec:hypercircle:qoi:review}

We briefly discuss flux reconstructions for the primal and the dual problem and settle some notation. 
We define the \emph{primal residual} $r_{h} \in H^{-1}(\Omega,\Gamma_N)$
and the \emph{dual residual} $s_{h} \in H^{-1}(\Omega,\Gamma_N)$
by 
\begin{subequations}
\label{math:residuals}
\begin{gather}
 \label{math:residuals:primal}
 r_{h} := F - \scrH_{\Omega,\Gamma_{D}} F + \divergence_{\Omega,\Gamma_{N}} A \grad u_{h}
 ,
 \\
 \label{math:residuals:dual}
 s_{h} := J - \scrH_{\Omega,\Gamma_{D}} J + \divergence_{\Omega,\Gamma_{N}} A \grad z_{h}
 .
\end{gather}
\end{subequations}
We introduce the residual flux equations 
% We assume that $\varrho_{h}, \varpi_{h} \in \bfL^{2}(\Omega)$ solve the residual flux equations 
\begin{gather}
 \label{math:residual:flux:equation}
 - \divergence_{\Omega,\Gamma_{N}} A \varrho_{h} = r_{h},
 \quad 
 - \divergence_{\Omega,\Gamma_{N}} A \varpi_{h} = s_{h}.
\end{gather}
The gradient error vector fields of the primal and the dual problem, 
\begin{gather}
 \label{math:optimalresidualfluxes}
 \varrho_{h}^{\rm opt} := \grad u - \grad u_{h},
 \quad 
  \varpi_{h}^{\rm opt} := \grad z - \grad z_{h},
\end{gather}
solve the residual flux equations. 
They are optimal solutions in the sense that their $\bfL^{2}$-norm
are minimal among the respective solution sets. 
% We can now write 
Hence 
\begin{gather}
 \label{math:optimalerroridentity}
 J( u - u_{h} ) = \left\langle \varrho_{h}^{\rm opt}, \varpi_{h}^{\rm opt} \right\rangle_{A} = F( z - z_{h} )
 .
\end{gather}
This identity expresses the error of interest 
as the product of the optimal residual flux reconstructions.
Of course, this result is practically inaccessible:
if the minimum norm residual flux reconstructions
$\varrho_{h}^{\rm opt}$ and $\varpi_{h}^{\rm opt}$ were known, 
then we could easily recover the true gradients $\grad u$ and $\grad z$
of the primal and the dual solution. 

In many applications, however, 
it is possible to compute finite element solutions 
$\varrho_{h} \in \RT^{r}_{-1}(\Omega)$ and $\varpi_{h} \in \RT^{r}_{-1}(\Omega)$ 
to the residual flux equations \eqref{math:residual:flux:equation}
that are close to the optimal solutions 
$\varrho_{h}^{\rm opt}$ and $\varpi_{h}^{\rm opt}$. 
We then have the Hodge-Helmholtz decompositions 
\begin{gather*}
 \varrho_{h} = \varrho_{h}^{\rm opt} + \theta_{h},
 \quad 
 \varpi_{h} = \varpi_{h}^{\rm opt} + \zeta_{h} 
\end{gather*}
for uniquely determined $\theta_{h}, \zeta_{h} \in \bfX(\Omega,\Gamma_{N},A)$.
% \mm{Q: $A$-orthogonal to which subspace? And why is $\varrho$ projected to
%   it equal to $\varrho_h^{\rm opt}$?}
Accordingly, 
we find by the $A$-orthogonality of the Hodge-Helmholtz decomposition that 
\begin{gather}
 \label{math:actualproduct} 
 \langle \varrho_{h}, \varpi_{h} \rangle_{A}
 =
 \left\langle \varrho_{h}^{\rm opt}, \varpi_{h}^{\rm opt} \right\rangle_{A}
 +
 \langle \theta_{h}, \zeta_{h} \rangle_{A}.
\end{gather}
Lastly, we recall that the flux reconstructions  
\begin{gather*}
 \sigma_{h} = \varrho_{h} + \grad u_{h},
 \quad 
 \tau  _{h} = \varpi _{h} + \grad z_{h} 
\end{gather*}
solve the primal and the dual flux equations
\begin{gather}
 \label{math:primalanddualfluxequations}
 - \divergence_{\Omega,\Gamma_{N}} A \sigma_{h} = F - \scrH_{\Omega,\Gamma_{D}} F,
 \quad 
 - \divergence_{\Omega,\Gamma_{N}} A \tau  _{h} = J - \scrH_{\Omega,\Gamma_{D}} J.
\end{gather}

\subsection{Heuristic Error Computation}
\label{subsec:hypercircle:qoi:heuristic}

We first recall some constant-free reliable error estimates that make use
of a (global) Cauchy-Schwarz inequality. The most immediate approach is to
trace back the error of interest to the error in the primal problem. Here
one uses 
\begin{align}
 \label{math:globalqoiestimate:naive}
 \begin{split}
  J( u - u_{h} ) 
  &= 
  \left\langle \nabla u - \nabla u_{h}, \nabla z \right\rangle_{A} 
  \\&\leq 
  \| \nabla u - \nabla u_{h} \|_{A} \cdot \| \nabla z \|_{A} 
  \leq 
  \| \varrho_{h} \|_{A} \cdot \| \tau_{h} \|_{A}
 \end{split}
\end{align}
A different variant uses the Galerkin orthogonality 
in the primal problem. We have 
\begin{align}
 \label{math:globalqoiestimate:better}
 \begin{split}
  J( u - u_{h} ) 
  &= 
  \left\langle \nabla u - \nabla u_{h}, \nabla z - \nabla z_{h} \right\rangle_{A} 
  \\&\leq 
  \| \nabla u - \nabla u_{h} \|_{A} \cdot \| \nabla z - \nabla z_{h} \|_{A} 
  \leq 
  \| \varrho_{h} \|_{A} \cdot \| \varpi_{h} \|_{A}
 \end{split}
\end{align}
Both estimates are reliable and constant free.
In applications we expect the error estimate $\| \varrho_{h} \|_{A} \cdot \| \varpi_{h} \|_{A}$
to converge with twice the rate of the error estimate $\| \varrho_{h} \|_{A} \cdot \| \tau_{h} \|_{A}$,
which indicates that error estimates for the primal problem alone 
lead to suboptimal error estimates. 

\begin{remark}
 The error estimates \eqref{math:globalqoiestimate:naive} and \eqref{math:globalqoiestimate:better}
 have been the starting point for several marking strategies in finite element methods. 
 Feischl, Praetorius, and van der Zee \cite{feischl2016abstract} 
 have investigated strategies of how to combine the information 
 of local error estimators for the primal and the dual problem 
 in order to drive goal-oriented adaptivity. 
\end{remark}

The most obvious problem with estimators based on the global Cauchy-Schwarz inequality 
is the overestimation of the true error 
since cancellation effects inside the integral $\langle \nabla u, \nabla z \rangle_{A}$
are not taken into account: if $\nabla u$ and $\nabla z$ are nearly $A$-orthogonal
to each other, then the overestimation will be substantial. 

This has motivated a variety of error estimators 
that approximate the error $J( u - u_{h} )$ by a computable integral. 
Consequently, these techniques give \emph{error approximations}
rather than \emph{error estimates}. 
On the one hand, 
these techniques are generally heuristic 
and may suffer from massive underestimation of the error \cite{nochetto2008safeguarded}.
On the other hand, 
the incorporation of cancellation effects 
achieves very accurate approximations for the true 
error of interest in applications.

We review some computable approximations 
for the error of interest $J( u - u_{h} )$. 
We first see 
\begin{align*}
 J( u - u_{h} )
 =
 \langle \nabla u - \nabla u_{h}, \nabla z \rangle_{A} 
 =
 \langle \sigma_{h} - \nabla u_{h}, \nabla \tau_{h} \rangle_{A} 
 -
 \langle \theta_{h}, \zeta_{h} \rangle_{A} 
 .
\end{align*}
Neglecting the term $\langle \theta_{h}, \zeta_{h} \rangle_{A}$, 
which is generally not computable we come to the error approximation 
\begin{gather}
 \label{math:qoi:approx:mp}
 J( u - u_{h} )
 \approx 
 \eta^{\varrho\tau}
 :=
 \langle \sigma_{h} - \nabla u_{h}, \tau_{h} \rangle_{A} 
 =
 \langle \varrho_{h}, \tau_{h} \rangle_{A} 
\end{gather}
and the local error indicators 
\begin{gather}
 \label{math:qoi:approx:mp:local}
 \eta^{\varrho\tau}_{K}
 := 
 \int_{K} \varrho_{h} \cdot A \tau_{h} \;\dif x, 
 \quad 
 K \in \Omega_{h}
 .
\end{gather}
It is intuitive that this error estimate can be improved 
by using the Galerkin orthogonality of the primal problem.
We have 
\begin{align*}
 J( u - u_{h} )
 &=
 \langle \nabla u - \nabla u_{h}, \nabla z  - \nabla z_{h} \rangle_{A} 
 \\&=
 \langle \sigma_{h} - \nabla u_{h}, \tau_{h} - \nabla z_{h} \rangle_{A} 
 -
 \langle \theta_{h}, \zeta_{h} \rangle_{A} 
 \\&=
 \langle \varrho_{h}, \varpi_{h} \rangle_{A} 
 -
 \langle \theta_{h}, \zeta_{h} \rangle_{A} 
 .
\end{align*}
Again neglecting the term $\langle \theta_{h}, \zeta_{h} \rangle_{A} $,
we consider the approximation 
\begin{gather}
 \label{math:qoi:approx:lm}
 J( u - u_{h} )
 \approx 
 \eta^{\varrho\varpi}
 :=
 \langle \varrho_{h}, \varpi_{h} \rangle_{A} 
 . 
\end{gather}
The corresponding local error indicators are 
\begin{gather}
 \label{math:qoi:approx:lm:local}
 \eta^{\varrho\varpi}_{K}
 :=
 \int_{K} \varrho_{h} \cdot A \varpi_{h} \;\dif x,
 \quad 
 K \in \Omega_{h}
 .
\end{gather}

\begin{remark}
 The error approximations \eqref{math:qoi:approx:mp}
 and \eqref{math:qoi:approx:lm} are, of course, the same, 
 but we state these different formulas because 
 they inspire the different local error indicators 
 \eqref{math:qoi:approx:mp:local} and \eqref{math:qoi:approx:lm:local}, 
 respectively. 
 In fact, 
 the error indicators \eqref{math:qoi:approx:mp:local}
 have been proposed by Mozolevski and Prudhomme \cite{mozolevski2015goal},
 who have investigated goal-oriented equilibrated error estimation 
 for a variety of discontinuous Galerkin finite element methods.
 They focus, however, on the special case of full elliptic regularity,
 where they numerically observed optimal adaptive convergence.
 Curiously, the local indicators \eqref{math:qoi:approx:lm:local}
 have apparently not been investigated before in the literature. 
 
 The size of the error term 
 $\langle \theta, \zeta \rangle_{A} = \langle \sigma_{h} - \nabla u_{h}, \tau_{h} - \nabla z \rangle_{A}$
 is decisive for the approximation quality of $\eta^{\varrho\tau}$ and $\eta^{\varrho\varpi}$. 
 When the flux reconstructions $\sigma_{h}$ and $\tau_{h}$ 
 are obtained by a mixed finite element method,
 then the product $\langle \theta, \zeta \rangle_{A}$ 
 can be controlled by convergence estimates for the mixed finite element method.
 Due to Inequality~\eqref{math:flux:globalvslocal}, 
 this is only slightly worse by a generic constant when 
 the localized flux reconstruction is used instead. 
 Since the polynomial degree (or the local mesh resolution) 
 in the localized flux reconstruction can be increased 
 with only constant factor in the computational effort, 
 we expect the error product $\langle \theta, \zeta \rangle_{A}$
 to be negligible in applications where $\nabla u$ and $\nabla z$
 feature sufficiently high regularity. 
 This gives reasonable hope that 
 $\eta^{\varrho\tau}$ and $\eta^{\varrho\varpi}$
 approximate the true error in many applications.
\end{remark}

The preceding local indicators are based on the approximation $\nabla z \approx \tau_{h}$,
which holds true up to $\bfX(\Omega,\Gamma_{N},A)$, i.e., 
the error of this approximation is orthogonal to gradients. 
Another possibility is to approximate $\nabla z$ by the gradient 
$\nabla z_{h}^{\ast}$ of an approximation $z_{h}^{\ast}$
of the dual solution $z$, i.e., $\nabla z \approx \nabla z_{h}^{\ast}$. 
Due to Galerkin orthogonality we should not choose any 
$z_{h}^{\ast} \in \bfQ^{r}(\Omega_{h},\Gamma_{D})$
since that would give a useless trivial approximation. 
A practically relevant method is constructing $z_{h}^{\ast}$
from $z_{h}$ in a higher order space on a coarser mesh (see \cite{richter2006}). 

With this background in mind, we observe that 
\begin{align*}
 J( u - u_{h} )
 &=
 \langle \nabla u - \nabla u_{h}, \nabla z \rangle_{A} 
 \\&=
 \langle \nabla u - \nabla u_{h}, \nabla z_{h}^{\ast} \rangle_{A} 
 -
 \langle \nabla u - \nabla u_{h}, \nabla z_{h}^{\ast} - \nabla z \rangle_{A} 
 \\&=
 \langle \sigma_{h} - \nabla u_{h}, \nabla z_{h}^{\ast} \rangle_{A} 
 -
 \langle \nabla u - \nabla u_{h}, \nabla z_{h}^{\ast} - \nabla z \rangle_{A} 
\end{align*}
If we assume that the product $\langle \nabla u - \nabla u_{h}, \nabla z_{h}^{\ast} - \nabla z \rangle_{A}$
is negligible, then this suggests the error approximation 
\begin{gather}
 \label{math:qoi:approx:dwr:falsch}
 J( u - u_{h} )
 \approx 
 \eta^{\rm I, \ast}
 :=
 \langle \sigma_{h} - \nabla u_{h}, \nabla z_{h}^{\ast} \rangle_{A} 
\end{gather}
and the corresponding error indicators 
\begin{gather}
 \label{math:qoi:approx:dwr:falsch:local}
 \eta^{\rm I, \ast}_{K}
 :=
 \int_{K} \varrho_{h} \cdot A \nabla z^{\ast}_{h} \;\dif x,
 \quad 
 K \in \Omega_{h}
 . 
\end{gather}
Again, the primal Galerkin orthogonality 
can be used for what promises to be an improvement. 
We then easily find in a manner similar as above that 
\begin{align*}
 J( u - u_{h} )
 =
 \langle \sigma_{h} - \nabla u_{h}, \nabla z_{h}^{\ast} - \nabla z_{h} \rangle_{A} 
 -
 \langle \nabla u - \nabla u_{h}, \nabla z_{h}^{\ast} - \nabla z \rangle_{A} 
 . 
\end{align*}
A promising error approximation is given by 
\begin{gather}
 \label{math:qoi:approx:dwr:richtig}
 J( u - u_{h} )
 \approx 
 \eta^{\rm II, \ast}
 :=
 \langle \sigma_{h} - \nabla u_{h}, \nabla z_{h}^{\ast} - \nabla z_{h} \rangle_{A} 
\end{gather}
with associated local error indicators 
\begin{gather}
 \label{math:qoi:approx:dwr:richtig:local}
 \eta^{\rm II, \ast}_{K}
 :=
 \int_{K} \varrho_{h} \cdot A \left( \nabla z^{\ast}_{h} - \nabla z_{h} \right) \;\dif x,
 \quad 
 K \in \Omega_{h}
 . 
\end{gather}

\begin{remark}
 \label{remark:dwr:vergleich}
 The error approximations $\eta^{\rm I,\ast}$ and $\eta^{\rm II,\ast}$
 and the associated error indicators $\eta^{\rm I,\ast}_{K}$ and $\eta^{\rm II,\ast}_{K}$
 are inspired by dual weighted residual methods \cite{becker2001optimal}.
 As a motivation for this class of goal-oriented error estimations 
 we first assume for the sake of simplicity 
 that $A$ is piecewise constant and $F \in L^{2}(\Omega)$. 
 We then observe that 
 \begin{align*}
  &
  \langle \nabla u - \nabla u_{h}, \nabla z^{\ast}_{h} - \nabla z_{h} \rangle_{A}
  \\
  &\qquad 
  =
  \sum_{K \in \Omega_{h}}
  \int_{K} 
  F\left( z^{\ast}_{h} - z_{h} \right)
  \;\dif x
  -
  \oint_{\partial K} \trace_{\partial K} \left( z^{\ast}_{h} - z_{h} \right) 
  \cdot 
  \trace_{N,K} A \nabla u_{h} \;\dif s.
 \end{align*}
 This has motivated the error approximation $\eta^{\rm DWR,\ast}$
 and the local error indicators $\eta^{\rm DWR,\ast}_{K}$
 of the dual weighted residual method,
 defined by 
 \begin{gather}
  \label{math:qoi:approx:dwr:rannacher}
  \eta^{\rm DWR,\ast}
  :=
  \sum_{K \in \Omega_{h}} \eta_{K}^{\rm DWR,\ast}
  \\
  \label{math:qoi:approx:dwr:rannacher:local}
  \eta^{\rm DWR,\ast}_{K}
  :=
  \int_{K} 
  F\left( z^{\ast}_{h} - z_{h} \right)
  \;\dif x
  -
  \oint_{\partial K} \trace_{\partial K} \left( z^{\ast}_{h} - z_{h} \right) 
  \cdot 
  \trace_{N,K} A \nabla u_{h} \;\dif s
  .
 \end{gather}
 Our error indicators $\eta_K^{\rm II,\ast}$ mimic the
 construction of the dual weighted residual error indicators $\eta_{K}^{\rm DWR,\ast}$
 but are not identical to them.  
 The error approximations $\eta^{\rm DWR,\ast}$ and $\eta^{\rm II,\ast}$, 
 however, are identical.
 
 Due to Galerkin orthogonality, the dual finite element approximation $z_{h}$
 can be removed from the definition of $\eta^{\rm DWR,\ast}$
 without changing the value;
 it makes a difference, however, whether the term $z_{h}$ appears in the
 local error indicators $\eta_K^{\rm DWR,\ast}$ or not. 
 It has long been known that the omission of $z_{h}$ 
 will generally lead to suboptimal adaptive marking strategies.
 This is analogous to our earlier observation 
 regarding the global Cauchy-Schwarz inequality. 
 In practice, we expect the indicators $\eta^{\rm I,\ast}_{K}$
 to perform similar to $\eta^{\varrho\tau}$
 and notably worse than $\eta^{\rm II,\ast}$.
\end{remark}

\begin{remark}
 There are several liberties in the exact definition of the local error indicators 
 but typically these have little effect on the overall performance in practice. 
 We have stated the error indicators as cellwise integrals
 as recommended by Becker and Rannacher (see Equation (3.19) of \cite{becker2001optimal})
 for the dual weighted residual method,
 but it also common to apply an additional Cauchy-Schwarz inequality 
 to the local integrals (see Equation (3.18) in the same reference).
 Furthermore, the dual Galerkin approximation $z_{h}$ 
 in the definition of our error indicators 
 can in practice be replaced by any other ``reasonable'' approximation 
 of $z$ in the finite element space $\bfQ^{r}(\Omega,\Gamma_{D})$. 
 Lastly, regarding specifically the error indicators 
 $\eta^{\rm II,\ast}_{K}$ and their relatives, 
 we mention that there are many possible ways 
 to pick an approximation $z_{h}^{\ast}$,
 see also Section~5 of \cite{becker2001optimal}.
 We utilize a higher order interpolation on a coarser mesh
 for our computational experiments in the next section.
\end{remark}

\section{Numerical Experiments}
\label{sec:experiments}

In this section we present a number of numerical experiments that document
properties of our proposed local patch-wise flux reconstruction, the global
equilibrated error estimator, and our heuristic goal-oriented error
estimators.
All numerical experiments have been computed with the help of the finite element library Deal.II
\cite{dealII85}. Solutions of linear problems (on global meshes and local
patches) have been computed with a direct solver \cite{suitesparse4}.

First, we demonstrate the efficiency of equilibrated error estimators when
the flux reconstruction is conducted locally (as discussed in Section~\ref{sec:flux})
or globally (akin to mixed finite element methods). We successively
document optimal convergence of an adaptive finite element method driven by
equilibrated error estimators.
Furthermore, we assess the performance of the error approximations and
indicators proposed in Section~\ref{sec:hypercircle:qoi}; our numerical
experiments document the efficiency under global uniform refinement 
and as drivers in goal-oriented adaptive finite element methods.

% macroscale and for the
% flux reconstruction) were computed with a direct solver
% \cite{suitesparse4}.

% In particular, we demonstrate robustness of
% the local flux reconstruction under local refinement (Section
% \ref{subsec:experiments-flux}) for two prototypical test cases: a regular,
% manufactured solution with local features; and the solution of the Laplace
% problem on a slit domain. We conclude this section with an estimator
% competition comparing all different variants of local indicators introduced
% in Section~\ref{sec:hypercircle:qoi}. 

%%%%%%%%%%%%%%%%%%%%%%%%%%%%%%%%%%%%%%%%%%%%%%%%%%%%%%%%%%%%%%%%%%%%%%%%%%%%%%%%
\subsection{Test Cases}
\label{subsec:experiments-testcases}

% We use two prototypical test cases in our experiments: 
% on the one hand, a regular manufactured solution 
% with local features over the unit square, 
% and on the other hand, 
% the solution of the Poisson problem over a slit domain
% for which we numerically compute a reference solution. 

Our numerical experiments have been carried out 
for two prototypical test cases. 
We consider the model Poisson equation \eqref{math:poisson} 
over two different domains: 
the unit square $\Omega$ and the slit domain $\Omega_{S}$.
\begin{gather}
 \label{math:domains}
 \Omega := (0,1)^{2},
 \quad
 \Omega_{S} := (-1,1)^{2} \setminus (0,1] \times \{0\}.
\end{gather}
In the \emph{first test case} 
we study the Poisson equation over the unit square $\Omega$ 
with the manufactured solution 
\begin{align}
  \label{eq:manufactured}
  u(x,y) &=
  \text{exp}\left(-100(x-\niceonehalf)^2-100(y-117/1000)^2\right)
\end{align}
and essential boundary conditions, chosen accordingly. 
In the \emph{second test case}
we solve the Poisson equation $-\Laplace u = 1$ over the slit domain $\Omega_{S}$
with homogeneous Dirichlet boundary conditions. 
For our numerical experiments we have computed a reference solution 
on a very fine mesh. 
Qualitative pictures of the two test solutions 
are given in Figure~\ref{fig:domain}. 
We note that the solution of the first test case is contained in $H^{2}(\Omega)$
whereas the solution of the second test case is contained in $H^{\frac{3}{2}}(\Omega_{S})$.

In the sequel, the initial meshes for the primal finite element methods
over $\Omega_{}$ and $\Omega_{S}$ 
are cubical meshes of resolution $16 \times 16$ in the first-order case;
we use a $15 \times 15$ mesh in the second-order case
so that the resulting finite element space has the same dimension. 

% Over the slit domain $\Omega_S$, we solve the
% Poisson problem with Dirichlet boundary conditions and a uniform right-hand
% side
% \begin{align}
%   \label{eq:constantrhs}
%   f = 1.
% \end{align}
% 
% we solve a Poisson problem with essential
% boundary conditions and the manufactured solution
% see Figure~\ref{fig:domain}. 

%%%%%%%%%%%%%%%%%%%%%%%%%%%%%%%%%%%%%%%%%%%%%%%%%%%%%%%%%%%%%%%%%%%%%%%%%%%%%%%%
\subsection{Efficiency of local flux reconstruction}
\label{subsec:experiments-flux}

\begin{figure}
  \centering
  \subfloat[]{\includegraphics[height=4cm]{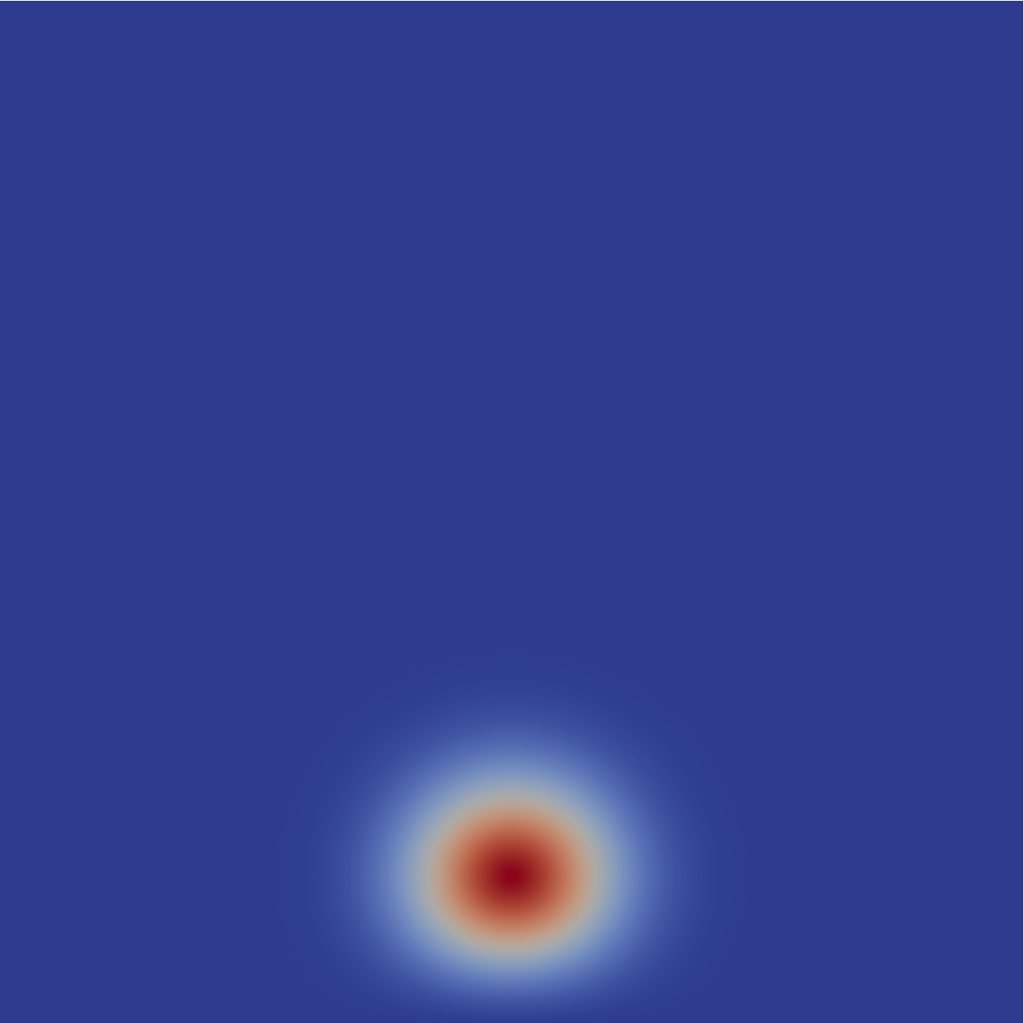}}
  \hspace{3em}
  \subfloat[]{\includegraphics[height=4cm]{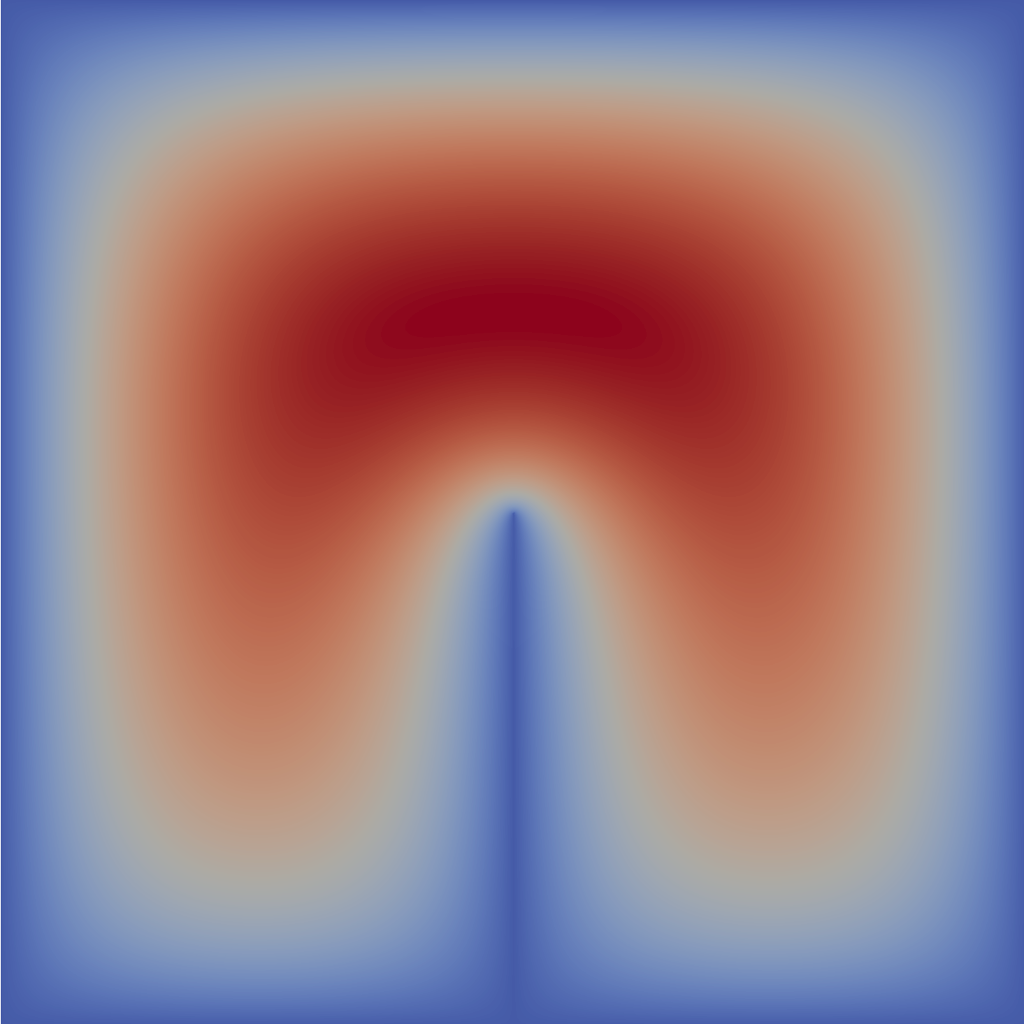}}
  \caption{
    The two considered test cases. 
    Picture (a) shows the manufactured solution
    \eqref{eq:manufactured} on the unit square, 
    Picture (b) shows the solution on the
    slit domain with constant unit right hand side.
  }
  \label{fig:domain}
  \vspace{1em}
  
  \footnotesize
  \subfloat[Manufactured solution, first and second order]{
    \begin{tabular} {r r c c c c c c}
      \toprule
      & & & & \multicolumn{2}{c}{Mixed method}
      & \multicolumn{2}{c}{Flux reconstruction}\\
      \cmidrule(lr){5-6}
      \cmidrule(lr){7-8}
      & \#Dofs
      & \multicolumn{2}{c}{$\|\nabla u_h - \nabla u\|_A^2$}
      & $\|\nabla u_h - \sigma_h\|_A^2$ & $I_{\text{eff},f}$
      & $\|\nabla u_h - \sigma_h\|_A^2$ & $I_{\text{eff},f}$
      \\[0.5em]
      1 &    289 & 2.84e-4 & 0.97 & 2.62e-4 & 0.921 & 2.97e-4 & 1.047 \\
      2 &   1089 & 7.35e-5 & 1.95 & 7.20e-5 & 0.980 & 7.54e-5 & 1.026 \\
      3 &   4225 & 1.85e-5 & 1.99 & 1.85e-5 & 0.996 & 1.87e-5 & 1.009 \\
      4 &  16641 & 4.63e-6 & 2.00 & 4.64e-6 & 1.003 & 4.66e-6 & 1.006 \\
      5 &  66049 & 1.15e-6 & 2.02 & 1.16e-6 & 1.016 & 1.16e-6 & 1.017 \\
      \\[0.5em]
      1 &    289 & 2.11e-4 & 1.81 & 1.75e-4 & 0.830 & 1.85e-4 & 0.880 \\
      2 &   1089 & 1.02e-5 & 4.37 & 9.67e-6 & 0.953 & 1.03e-5 & 1.017 \\
      3 &   4225 & 6.69e-7 & 3.92 & 6.61e-7 & 0.988 & 6.76e-7 & 1.010 \\
      4 &  16641 & 4.24e-8 & 3.98 & 4.23e-8 & 0.997 & 4.26e-8 & 1.003 \\
      5 &  66049 & 2.66e-9 & 3.99 & 2.66e-9 & 0.999 & 2.66e-9 & 1.001 \\
      \bottomrule
    \end{tabular}}

  \subfloat[Slit domain, first and second order]{
    \begin{tabular} {r r c c c c c c}
      \toprule
      & & & & \multicolumn{2}{c}{Mixed method}
      & \multicolumn{2}{c}{Flux reconstruction}\\
      \cmidrule(lr){5-6}
      \cmidrule(lr){7-8}
      & \#Dofs
      & \multicolumn{2}{c}{$\|\nabla u_h - \nabla u\|_A^2$}
      & $\|\nabla u_h - \sigma_h\|_A^2$ & $I_{\text{eff},f}$
      & $\|\nabla u_h - \sigma_h\|_A^2$ & $I_{\text{eff},f}$
      \\[0.5em]
      1 &   1105 & 2.46e-4 & 1.34 & 3.09e-4 & 1.255 & 3.99e-4 & 1.617 \\
      2 &   4257 & 1.06e-4 & 1.22 & 1.39e-4 & 1.309 & 1.83e-4 & 1.725 \\
      3 &  16705 & 4.76e-5 & 1.16 & 6.55e-5 & 1.376 & 8.73e-5 & 1.835 \\
      4 &  66177 & 2.14e-5 & 1.16 & 3.18e-5 & 1.488 & 4.26e-5 & 1.997 \\
      5 & 263425 & 8.96e-6 & 1.25 & 1.56e-5 & 1.745 & 2.11e-5 & 2.351
      \\[0.5em]
      1 &   1105 & 1.19e-4 & 1.03 & 1.82e-4 & 1.529 & 2.75e-4 & 2.314 \\
      2 &   4257 & 5.82e-5 & 1.03 & 9.05e-5 & 1.556 & 1.37e-4 & 2.349 \\
      3 &  16705 & 2.81e-5 & 1.05 & 4.52e-5 & 1.609 & 6.82e-5 & 2.425 \\
      4 &  66177 & 1.31e-5 & 1.10 & 2.26e-5 & 1.723 & 3.41e-5 & 2.597 \\
      5 & 263425 & 5.62e-6 & 1.22 & 1.13e-5 & 2.011 & 1.70e-5 & 3.031 \\
      \bottomrule
    \end{tabular}}
    \vspace{1em}
%   \captionof{table}{Example Table (floating)}
  \captionof{table}{Parameter study demonstrating the efficiency of the equilibrated
    error estimator for the two considered test cases of a manufactured
    solution (a) and a solution over the slit domain (b).}
  \label{table:test1:uniform}
\end{figure}

As a first numerical experiment we present a short parameter study
demonstrating the robustness of the proposed local flux reconstruction. 
% For this, we consider model problem~
We solve both problems with tensor product Lagrange elements
of order one and two
and compute a flux reconstruction in the Raviart-Thomas space of 
order two or three, respectively. 
We compute the flux reconstruction 
either 
by solving the adjoint mixed problem 
or by patch-wise flux reconstruction. 
The efficiency $I_{\text{eff},f}$ of the respective equilibrated error estimate with
reconstructed flux $\sigma_{h}$ is assessed with the ratio
in Equation \eqref{math:efficiency}.
% , which gives 
% \begin{gather*}
%  I_{\text{eff},f} =
%  \dfrac{ \| \grad u_h - \sigma_h \|^{2}_{A} }{ \| \grad u_h - \grad u \|^{2}_{A} }
%  =
%  1
%  +
%  \dfrac{ \| \theta_h \|^{2}_{A} }{ \| \grad u_h - \grad u \|^{2}_{A} }
%  .
% \end{gather*}

Table~\ref{table:test1:uniform} gives computational results 
for first and second order Lagrange elements.
In the first test case with a manufactured solution, 
the efficiency indices of both methods of flux reconstruction are close to optimal,
i.e., $I_{\text{eff},f} \approx 1$. 
That the data of the problem are not cellwise polynomial 
is reflected in the slight underestimation of the error in
initial phases for the flux reconstruction via mixed methods.
Both estimators appear to be asymptotically exact;
we attribute this to the high regularity of $\nabla u$
and the higher polynomial order of the flux reconstruction, 
which make the term $\nabla u - \sigma$ converge to zero 
faster than $\nabla u - \nabla u_{h}$. 

In the second test case over the slit domain,
the efficiency is generally worse with $I_{\text{eff},f}\approx 1.5 - 3.0$ 
for the tested refinement levels. 
% The efficiency seems to detoriate over time, 
% which is in accordance to our theoretical results:
% whereas the gradient error of the elliptic finite element method 
% converges with order $h$, the error in the flux variable of 
% a mixed finite element method converges only with order $h^{\onehalf}$
% on such a singular domain; 
% hence the latter error term dominates the equilibrated error estimate. 
We attribute this to the corner singularity
limiting the regularity of $\nabla u$. 
We now also observe the predicted lower efficiency
with local instead of global flux reconstruction 
(see Inequality~\eqref{math:flux:globalvslocal})
and the effect of increased generic constants 
for higher polynomial order. 

%%%%%%%%%%%%%%%%%%%%%%%%%%%%%%%%%%%%%%%%%%%%%%%%%%%%%%%%%%%%%%%%%%%%%%%%%%%%%%%%
\subsection{Energy-Oriented Adaptive Refinement}
\label{subsec:experiments-flux-adaptive}

Additionally, the equilibrated error estimator has been tested 
in the second test case as a driver for adaptive mesh refinement. 
The local indicators are computed in accordance with \eqref{math:localerrorindicator},
once with the flux reconstructed from a mixed finite element method 
and once with the locally reconstructed flux. 
We use a \emph{fixed fraction} marking strategy where in each
refinement step the 33\,\% of cells with highest indicator values are
marked for refinement. 
Table~\ref{table:test1:adaptive} displays the results 
with first-order Lagrange elements.
The obtained efficiency indices
$I_{\text{eff},f} \approx 1.4$ are considerably better than the
corresponding case with uniform refinement (see Table~\ref{table:test1:uniform}).
% Again, asymptotic detoriation of efficiency can be observed 
% in the case of such pathological domains.
The estimator based on the local flux reconstruction has slightly worse
efficiency indices but performs similarly well 
as the estimator based on the mixed method.

\begin{table}[!t]
  \center
  \footnotesize
  \begin{tabular} {r r c c c c c c}
    \toprule
    & & & & \multicolumn{2}{c}{Mixed method}
    & \multicolumn{2}{c}{Flux reconstruction}\\
    \cmidrule(lr){5-6}
    \cmidrule(lr){7-8}
    & \#Dofs
    & \multicolumn{2}{c}{$\|\nabla u_h - \nabla u\|_A^2$}
    & $\|\nabla u_h - \sigma_h\|_A^2$ & $I_{\text{eff},f}$
    & $\|\nabla u_h - \sigma_h\|_A^2$ & $I_{\text{eff},f}$
    \\[0.5em]
    1 &    311 & 4.80e-4 & 1.20 & 5.44e-4 & 1.132 & 6.32e-4 & 1.315 \\
    2 &    628 & 2.82e-4 & 0.77 & 3.16e-4 & 1.118 & 3.58e-4 & 1.270 \\
    3 &   1159 & 1.27e-4 & 1.16 & 1.45e-4 & 1.142 & 1.65e-4 & 1.304 \\
    4 &   2333 & 6.97e-5 & 0.86 & 8.02e-5 & 1.150 & 9.07e-5 & 1.302 \\
    5 &   4393 & 3.05e-5 & 1.19 & 3.72e-5 & 1.219 & 4.23e-5 & 1.388 \\
    \bottomrule\\
  \end{tabular}
  \caption{Parameter study demonstrating the efficiency of the local flux
    reconstruction for the second test case of a slit domain.}
  \label{table:test1:adaptive}
\end{table}

%%%%%%%%%%%%%%%%%%%%%%%%%%%%%%%%%%%%%%%%%%%%%%%%%%%%%%%%%%%%%%%%%%%%%%%%%%%%%%%%
\subsection{Goal-Oriented Error Estimation}
\label{subsec:experiments-goaloriented-uniform}

\begin{table}[!t]
  \center
  \footnotesize
  \subfloat[Manufactured solution, first and second order]{
    \begin{tabular} {r r c c c c c c c c}
      \toprule
      & & &
      & \multicolumn{2}{c}{$\eta^{\varrho\varpi}$}
      & \multicolumn{2}{c}{$\eta^{\rm II,\ast}$}
      & \multicolumn{2}{c}{$\eta^{\varrho\tau}$}\\
      \cmidrule(lr){5-6}
      \cmidrule(lr){7-8}
      \cmidrule(lr){9-10}
      & \#Dofs
      & \multicolumn{2}{c}{$\big|J(u)-J(u_h)\big|$}
          & $I_{\text{eff}}$ & $I_{\text{osc}}$
          & $I_{\text{eff}}$ & $I_{\text{osc}}$
          & $I_{\text{eff}}$ & $I_{\text{osc}}$
      \\[0.5em]
      1 &    289 &  8.38e-4 & 0.89 &  1.043 & 1.185 &  0.274 & 2.704 &  2.699 & 1.122  \\
      2 &   1089 &  2.17e-4 & 1.95 &  1.022 & 1.183 &  0.899 & 1.229 &  2.831 & 1.211  \\
      3 &   4225 &  5.45e-5 & 1.99 &  1.007 & 1.198 &  0.976 & 1.213 &  2.827 & 1.242  \\
      4 &  16641 &  1.36e-5 & 2.00 &  1.006 & 1.201 &  0.997 & 1.205 &  2.832 & 1.257  \\
      5 &  66049 &  3.37e-6 & 2.02 &  1.016 & 1.202 &  1.014 & 1.203 &  2.865 & 1.260
      \\[0.5em]
      1 &    289 &  5.65e-4 & 2.06 &  0.893 & 1.000 &  0.183 & 1.265 &  0.462 & 2.440  \\
      2 &   1089 &  2.22e-5 & 4.67 &  1.021 & 1.001 &  1.242 & 1.001 &  1.936 & 4.089  \\
      3 &   4225 &  1.43e-6 & 3.95 &  1.008 & 1.044 &  0.903 & 1.081 &  1.965 & 5.601  \\
      4 &  16641 &  9.15e-8 & 3.96 &  0.989 & 1.049 &  0.985 & 1.057 &  1.941 & 5.923  \\
      5 &  66049 &  5.64e-9 & 4.02 &  1.003 & 1.047 &  1.030 & 1.051 &  1.974 & 6.037  \\
      \bottomrule
    \end{tabular}}

  \subfloat[Slit domain, first and second order]{
    \begin{tabular} {r r c c c c c c c c}
      \toprule
      & & &
      & \multicolumn{2}{c}{$\eta^{\varrho\varpi}$}
      & \multicolumn{2}{c}{$\eta^{\rm II,\ast}$}
      & \multicolumn{2}{c}{$\eta^{\varrho\tau}$}\\
      \cmidrule(lr){5-6}
      \cmidrule(lr){7-8}
      \cmidrule(lr){9-10}
      & \#Dofs
      & \multicolumn{2}{c}{$\big|J(u)-J(u_h)\big|$}
          & $I_{\text{eff}}$ & $I_{\text{osc}}$
          & $I_{\text{eff}}$ & $I_{\text{osc}}$
          & $I_{\text{eff}}$ & $I_{\text{osc}}$ \\[0.5em]
          1 &   1105 &  3.08e-4 & 1.29 &  1.639 & 1.015 &  0.454 & 1.062 &  1.284 & 1.126  \\
          2 &   4257 &  1.38e-4 & 1.16 &  1.699 & 1.008 &  0.391 & 1.041 &  1.308 & 1.111  \\
          3 &  16705 &  6.51e-5 & 1.08 &  1.733 & 1.004 &  0.354 & 1.023 &  1.321 & 1.107  \\
          4 &  66177 &  3.17e-5 & 1.04 &  1.746 & 1.002 &  0.333 & 1.014 &  1.322 & 1.105  \\
          5 & 263425 &  1.58e-5 & 1.01 &  1.739 & 1.001 &  0.319 & 1.008 &  1.313 & 1.105
      \\[0.5em]
          1 &   1105 &  1.57e-4 & 0.92 &  2.164 & 1.001 &  0.115 & 1.093 &  1.836 & 1.709  \\
          2 &   4257 &  7.88e-5 & 0.99 &  2.045 & 1.000 &  0.108 & 1.066 &  1.736 & 1.706  \\
          3 &  16705 &  3.95e-5 & 1.00 &  1.859 & 1.000 &  0.099 & 1.063 &  1.578 & 1.710  \\
          4 &  66177 &  1.99e-5 & 0.99 &  1.572 & 1.000 &  0.083 & 1.062 &  1.334 & 1.715  \\
          5 & 263425 &  1.01e-5 & 0.98 &  1.201 & 1.000 &  0.064 & 1.062 &  1.019 & 1.718  \\
      \bottomrule
    \end{tabular}}
    \vspace{1em}
    \caption{Efficiency $I_{\text{eff}}$ and oscillatory behavior
      $I_{\text{osc}}$ of different error estimators for first and second
      polynomial order and for the two test cases of (a) a manufactured solution on
      the unit square and (b) a solution on the slit-domain. 
%       The indices are computed according to \eqref{eq:errorestimatemeasures}.
      }
  \label{table:test2}
\end{table}

Our next computational experiments assess the performance 
of our proposed goal-oriented error estimators.
We consider the error approximations and indicators 
\begin{align*}
  \begin{aligned}
    \eta^{\varrho\varpi} &=\sum_{ K \in \Omega_{h}} \eta^{\varrho\varpi}_K,
    &\qquad
    \eta^{\varrho\varpi}_K &= \int_K\langle\varrho_h,\varpi_h\rangle_A,
    \\
    \eta^{\rm II,\ast} &=\sum_{ K \in \Omega_{h}} \eta^{\rm II,\ast}_K,
    &\qquad
    \eta^{\rm II,\ast}_K &= \int_K\langle\varrho_h, z^{\ast}_{h} - z_h\rangle_A,
    \\
    \eta^{\varrho\tau} &=\sum_{ K \in \Omega_{h}} \eta^{\varrho\tau}_K,
    &\qquad
    \eta^{\varrho\tau}_K &= \int_K\langle\varrho_h,\tau_h\rangle_A.
  \end{aligned}
\end{align*}
The error indicators $\eta^{\varrho\varpi}$ and $\eta^{\varrho\tau}$
utilize the local flux reconstructions in the 
primal and the dual problems.
% {\matthias
The error indicator $\eta^{\rm II,\ast}$ uses only the local flux
reconstruction of the primal problem 
but employs an approximation $z^{\ast}_{h}$ of the dual solution 
that is reconstructed by a well-known postprocessing 
technique---a patch-wise higher-order interpolation 
to a coarser mesh, $z^{\ast}_{h} = \Pi^{(2r)}_{2h} z_h$ 
in the notation of \cite{richter2006}.
% }

% Two different measures assess the indicators quantitatively. 
Whenever $\eta^\ast = \sum_{ K \in \Omega_{h}}\eta^\ast_K$ 
is one of the error estimators with local indicators proposed 
in Section~\ref{sec:hypercircle:qoi},
we define the \emph{efficiency index} $I_{\text{eff}}$ 
and the \emph{oscillation index} $I_{\text{osc}} \geq 1$
by
\begin{align}
  \label{eq:errorestimatemeasures}
  I_{\text{eff}} := \frac{|\eta^\ast|}{\big|J(u)-J(u_h)\big|},
  \quad
  I_{\text{osc}} :=
  \frac1{|\eta^\ast|}\Big(\sum_{ K \in \Omega_{h}}\big|\eta^\ast_K\big|\Big)
  .
\end{align}
Whereas $I_{\text{eff}}$ measures the total efficiency of the error estimate,
the quantity $I_{\text{osc}}$ measures the oscillatory behavior of the local
indicators.
For our computational experiments we extend the two different test cases 
by goal functionals. Specifically, the quantity of interest is defined 
by small regularized point evaluations at $(0.5, 0.117)$ in the first case 
and at $(0.25,0.75)$ in the second test case. 

We have assessed the performance of the error approximations and indicators 
under uniform refinement, again considering 
first and second order Lagrange elements. 
For the sake of brevity we consider only the (practically interesting) 
case of local flux reconstruction.

The results are given in Table~\ref{table:test2}.
The proposed error approximation $\eta^{\varrho\varpi}$ 
performs very well and appears to behave asymptotically exact 
in both test cases and both polynomial orders. 
By comparison, the error approximation $\eta^{\rm II,\ast}$,
which coincides with the error approximation 
of the dual weighted residual method 
(see Remark~\ref{remark:dwr:vergleich}),
performs reasonably on sufficiently fine meshes 
in the first (regular) test case  
but fails to approximate the true error 
in the second test case over the slit domain. 
The third error indicator $\eta^{\varrho\tau}$ 
gives reliable upper bounds in the experiments 
but consistently overestimates the error 
in the regular first test case
when compared to $\eta^{\varrho\varpi}$ and $\eta^{\rm II,\ast}$. 
% 
% %
% 1)  MP ist mehr oscillatory, in particular for higher order 
% 2)  MP ist weniger effizient bei problemen mit hoher reguliartät 
% 2)  DWR (!!!) säuft ab auf der slit domain 
% 
% 
Finally, 
it is worth mentioning that $\eta^{\varrho\tau}$ 
consistently exhibits a pronounced oscillatory behavior
whereas the oscillation index of $\eta^{\varrho\varpi}$ and $\eta^{\rm II,\ast}$
appear to be asymptotically optimal or near-optimal, respectively.

% As a computational experiment we have again considered our two different
% Poisson problems above. As quantity of interest we use a small, regularized
% ``Dirac-delta'' at $(0.5, 0.117)$ for the manufactured solution, and
% $(0.25,0.75)$ for the slit-domain. We have tested first and second order
% Lagrange elements for both problems under uniform refinement, as well as
% first order Lagrange elements under local refinement driven by the
% goal-oriented local error indicators listed above. 

\subsection{Goal-Oriented Estimator Competition}
\label{subsec:experiments-goaloriented-competition}

\begin{figure}[p]
  \centering
  \subfloat[unit square]{\begin{tikzpicture}[gnuplot,scale=0.8]
%% generated with GNUPLOT 5.0p6 (Gentoo revision r0) (Lua 5.1; terminal rev. 99, script rev. 100)
%% Sun 23 Jul 2017 08:35:53 PM CDT
\path (0.000,0.000) rectangle (12.500,8.750);
\gpcolor{color=gp lt color border}
\gpsetlinetype{gp lt border}
\gpsetdashtype{gp dt solid}
\gpsetlinewidth{1.00}
\draw[gp path] (2.240,0.985)--(2.330,0.985);
\draw[gp path] (11.947,0.985)--(11.857,0.985);
\draw[gp path] (2.240,1.404)--(2.330,1.404);
\draw[gp path] (11.947,1.404)--(11.857,1.404);
\gpcolor{color=gp lt color axes}
\gpsetlinetype{gp lt axes}
\gpsetdashtype{gp dt axes}
\gpsetlinewidth{0.50}
\draw[gp path] (2.240,1.603)--(11.947,1.603);
\gpcolor{color=gp lt color border}
\gpsetlinetype{gp lt border}
\gpsetdashtype{gp dt solid}
\gpsetlinewidth{1.00}
\draw[gp path] (2.240,1.603)--(2.420,1.603);
\draw[gp path] (11.947,1.603)--(11.767,1.603);
\node[gp node right] at (2.056,1.603) {$10^{-5}$};
\draw[gp path] (2.240,2.221)--(2.330,2.221);
\draw[gp path] (11.947,2.221)--(11.857,2.221);
\draw[gp path] (2.240,3.038)--(2.330,3.038);
\draw[gp path] (11.947,3.038)--(11.857,3.038);
\draw[gp path] (2.240,3.457)--(2.330,3.457);
\draw[gp path] (11.947,3.457)--(11.857,3.457);
\gpcolor{color=gp lt color axes}
\gpsetlinetype{gp lt axes}
\gpsetdashtype{gp dt axes}
\gpsetlinewidth{0.50}
\draw[gp path] (2.240,3.656)--(11.947,3.656);
\gpcolor{color=gp lt color border}
\gpsetlinetype{gp lt border}
\gpsetdashtype{gp dt solid}
\gpsetlinewidth{1.00}
\draw[gp path] (2.240,3.656)--(2.420,3.656);
\draw[gp path] (11.947,3.656)--(11.767,3.656);
\node[gp node right] at (2.056,3.656) {$10^{-4}$};
\draw[gp path] (2.240,4.274)--(2.330,4.274);
\draw[gp path] (11.947,4.274)--(11.857,4.274);
\draw[gp path] (2.240,5.092)--(2.330,5.092);
\draw[gp path] (11.947,5.092)--(11.857,5.092);
\draw[gp path] (2.240,5.511)--(2.330,5.511);
\draw[gp path] (11.947,5.511)--(11.857,5.511);
\gpcolor{color=gp lt color axes}
\gpsetlinetype{gp lt axes}
\gpsetdashtype{gp dt axes}
\gpsetlinewidth{0.50}
\draw[gp path] (2.240,5.710)--(11.947,5.710);
\gpcolor{color=gp lt color border}
\gpsetlinetype{gp lt border}
\gpsetdashtype{gp dt solid}
\gpsetlinewidth{1.00}
\draw[gp path] (2.240,5.710)--(2.420,5.710);
\draw[gp path] (11.947,5.710)--(11.767,5.710);
\node[gp node right] at (2.056,5.710) {$10^{-3}$};
\draw[gp path] (2.240,6.328)--(2.330,6.328);
\draw[gp path] (11.947,6.328)--(11.857,6.328);
\draw[gp path] (2.240,7.145)--(2.330,7.145);
\draw[gp path] (11.947,7.145)--(11.857,7.145);
\draw[gp path] (2.240,7.564)--(2.330,7.564);
\draw[gp path] (11.947,7.564)--(11.857,7.564);
\gpcolor{color=gp lt color axes}
\gpsetlinetype{gp lt axes}
\gpsetdashtype{gp dt axes}
\gpsetlinewidth{0.50}
\draw[gp path] (2.240,7.763)--(9.007,7.763);
\draw[gp path] (11.763,7.763)--(11.947,7.763);
\gpcolor{color=gp lt color border}
\gpsetlinetype{gp lt border}
\gpsetdashtype{gp dt solid}
\gpsetlinewidth{1.00}
\draw[gp path] (2.240,7.763)--(2.420,7.763);
\draw[gp path] (11.947,7.763)--(11.767,7.763);
\node[gp node right] at (2.056,7.763) {$10^{-2}$};
\draw[gp path] (2.240,8.381)--(2.330,8.381);
\draw[gp path] (11.947,8.381)--(11.857,8.381);
\draw[gp path] (2.240,0.985)--(2.240,1.075);
\draw[gp path] (2.240,8.381)--(2.240,8.291);
\draw[gp path] (2.873,0.985)--(2.873,1.075);
\draw[gp path] (2.873,8.381)--(2.873,8.291);
\draw[gp path] (3.323,0.985)--(3.323,1.075);
\draw[gp path] (3.323,8.381)--(3.323,8.291);
\draw[gp path] (3.671,0.985)--(3.671,1.075);
\draw[gp path] (3.671,8.381)--(3.671,8.291);
\draw[gp path] (3.956,0.985)--(3.956,1.075);
\draw[gp path] (3.956,8.381)--(3.956,8.291);
\draw[gp path] (4.197,0.985)--(4.197,1.075);
\draw[gp path] (4.197,8.381)--(4.197,8.291);
\draw[gp path] (4.405,0.985)--(4.405,1.075);
\draw[gp path] (4.405,8.381)--(4.405,8.291);
\draw[gp path] (4.589,0.985)--(4.589,1.075);
\draw[gp path] (4.589,8.381)--(4.589,8.291);
\gpcolor{color=gp lt color axes}
\gpsetlinetype{gp lt axes}
\gpsetdashtype{gp dt axes}
\gpsetlinewidth{0.50}
\draw[gp path] (4.754,0.985)--(4.754,8.381);
\gpcolor{color=gp lt color border}
\gpsetlinetype{gp lt border}
\gpsetdashtype{gp dt solid}
\gpsetlinewidth{1.00}
\draw[gp path] (4.754,0.985)--(4.754,1.165);
\draw[gp path] (4.754,8.381)--(4.754,8.201);
\node[gp node center] at (4.754,0.677) {$100$};
\draw[gp path] (5.837,0.985)--(5.837,1.075);
\draw[gp path] (5.837,8.381)--(5.837,8.291);
\draw[gp path] (6.470,0.985)--(6.470,1.075);
\draw[gp path] (6.470,8.381)--(6.470,8.291);
\draw[gp path] (6.919,0.985)--(6.919,1.075);
\draw[gp path] (6.919,8.381)--(6.919,8.291);
\draw[gp path] (7.268,0.985)--(7.268,1.075);
\draw[gp path] (7.268,8.381)--(7.268,8.291);
\draw[gp path] (7.553,0.985)--(7.553,1.075);
\draw[gp path] (7.553,8.381)--(7.553,8.291);
\draw[gp path] (7.793,0.985)--(7.793,1.075);
\draw[gp path] (7.793,8.381)--(7.793,8.291);
\draw[gp path] (8.002,0.985)--(8.002,1.075);
\draw[gp path] (8.002,8.381)--(8.002,8.291);
\draw[gp path] (8.186,0.985)--(8.186,1.075);
\draw[gp path] (8.186,8.381)--(8.186,8.291);
\gpcolor{color=gp lt color axes}
\gpsetlinetype{gp lt axes}
\gpsetdashtype{gp dt axes}
\gpsetlinewidth{0.50}
\draw[gp path] (8.350,0.985)--(8.350,8.381);
\gpcolor{color=gp lt color border}
\gpsetlinetype{gp lt border}
\gpsetdashtype{gp dt solid}
\gpsetlinewidth{1.00}
\draw[gp path] (8.350,0.985)--(8.350,1.165);
\draw[gp path] (8.350,8.381)--(8.350,8.201);
\node[gp node center] at (8.350,0.677) {$1000$};
\draw[gp path] (9.433,0.985)--(9.433,1.075);
\draw[gp path] (9.433,8.381)--(9.433,8.291);
\draw[gp path] (10.066,0.985)--(10.066,1.075);
\draw[gp path] (10.066,8.381)--(10.066,8.291);
\draw[gp path] (10.516,0.985)--(10.516,1.075);
\draw[gp path] (10.516,8.381)--(10.516,8.291);
\draw[gp path] (10.864,0.985)--(10.864,1.075);
\draw[gp path] (10.864,8.381)--(10.864,8.291);
\draw[gp path] (11.149,0.985)--(11.149,1.075);
\draw[gp path] (11.149,8.381)--(11.149,8.291);
\draw[gp path] (11.390,0.985)--(11.390,1.075);
\draw[gp path] (11.390,8.381)--(11.390,8.291);
\draw[gp path] (11.598,0.985)--(11.598,1.075);
\draw[gp path] (11.598,8.381)--(11.598,8.291);
\draw[gp path] (11.782,0.985)--(11.782,1.075);
\draw[gp path] (11.782,8.381)--(11.782,8.291);
\gpcolor{color=gp lt color axes}
\gpsetlinetype{gp lt axes}
\gpsetdashtype{gp dt axes}
\gpsetlinewidth{0.50}
\draw[gp path] (11.947,0.985)--(11.947,8.381);
\gpcolor{color=gp lt color border}
\gpsetlinetype{gp lt border}
\gpsetdashtype{gp dt solid}
\gpsetlinewidth{1.00}
\draw[gp path] (11.947,0.985)--(11.947,1.165);
\draw[gp path] (11.947,8.381)--(11.947,8.201);
\node[gp node center] at (11.947,0.677) {$10000$};
\draw[gp path] (2.240,8.381)--(2.240,0.985)--(11.947,0.985)--(11.947,8.381)--cycle;
\node[gp node center,rotate=-270] at (0.246,4.683) {$|J(u)-J(u_h)|$};
\node[gp node center] at (7.093,0.215) {\#dofs primal};
\node[gp node right] at (10.479,8.047) {uniform};
\gpcolor{rgb color={0.580,0.000,0.827}}
\draw[gp path] (10.663,8.047)--(11.579,8.047);
\draw[gp path] (2.589,7.827)--(4.425,6.100)--(6.412,5.552)--(8.484,4.346)--(10.601,3.115)%
  --(11.947,2.338);
\gpcolor{color=gp lt color border}
\node[gp node right] at (10.479,7.239) {$\eta^{\varrho\varpi}_K$};
\gpcolor{rgb color={0.000,0.620,0.451}}
\draw[gp path] (10.663,7.239)--(11.579,7.239);
\draw[gp path] (2.589,7.827)--(3.607,6.082)--(4.641,5.564)--(5.740,4.420)--(6.900,3.425)%
  --(7.956,2.562)--(8.987,2.043)--(10.036,1.226);
\gpcolor{color=gp lt color border}
\node[gp node right] at (10.479,6.431) {$\eta^{\text{II},\ast}_K$};
\gpcolor{rgb color={0.337,0.706,0.914}}
\draw[gp path] (10.663,6.431)--(11.579,6.431);
\draw[gp path] (2.589,7.827)--(3.640,6.102)--(4.624,5.563)--(5.867,4.625)--(6.864,3.569)%
  --(7.880,2.634)--(8.934,2.093)--(9.960,1.280);
\gpcolor{color=gp lt color border}
\node[gp node right] at (10.479,5.623) {$\eta^{\varrho\tau}_K$};
\gpcolor{rgb color={0.902,0.624,0.000}}
\draw[gp path] (10.663,5.623)--(11.579,5.623);
\draw[gp path] (2.589,7.827)--(3.640,6.080)--(4.641,5.659)--(5.655,4.682)--(6.776,3.552)%
  --(7.990,3.028)--(9.139,2.324)--(10.325,1.870);
\gpcolor{color=gp lt color border}
\draw[gp path] (2.240,8.381)--(2.240,0.985)--(11.947,0.985)--(11.947,8.381)--cycle;
%% coordinates of the plot area
\gpdefrectangularnode{gp plot 1}{\pgfpoint{2.240cm}{0.985cm}}{\pgfpoint{11.947cm}{8.381cm}}
\end{tikzpicture}
%% gnuplot variables
}
  
  \subfloat[slit domain]{\begin{tikzpicture}[gnuplot,scale=0.8]
%% generated with GNUPLOT 5.0p6 (Gentoo revision r0) (Lua 5.1; terminal rev. 99, script rev. 100)
%% Sun 23 Jul 2017 08:17:51 PM CDT
\path (0.000,0.000) rectangle (12.500,8.750);
\gpcolor{color=gp lt color border}
\gpsetlinetype{gp lt border}
\gpsetdashtype{gp dt solid}
\gpsetlinewidth{1.00}
\draw[gp path] (2.240,0.985)--(2.330,0.985);
\draw[gp path] (11.947,0.985)--(11.857,0.985);
\draw[gp path] (2.240,1.419)--(2.330,1.419);
\draw[gp path] (11.947,1.419)--(11.857,1.419);
\draw[gp path] (2.240,1.727)--(2.330,1.727);
\draw[gp path] (11.947,1.727)--(11.857,1.727);
\draw[gp path] (2.240,1.966)--(2.330,1.966);
\draw[gp path] (11.947,1.966)--(11.857,1.966);
\draw[gp path] (2.240,2.161)--(2.330,2.161);
\draw[gp path] (11.947,2.161)--(11.857,2.161);
\draw[gp path] (2.240,2.326)--(2.330,2.326);
\draw[gp path] (11.947,2.326)--(11.857,2.326);
\draw[gp path] (2.240,2.469)--(2.330,2.469);
\draw[gp path] (11.947,2.469)--(11.857,2.469);
\draw[gp path] (2.240,2.595)--(2.330,2.595);
\draw[gp path] (11.947,2.595)--(11.857,2.595);
\gpcolor{color=gp lt color axes}
\gpsetlinetype{gp lt axes}
\gpsetdashtype{gp dt axes}
\gpsetlinewidth{0.50}
\draw[gp path] (2.240,2.708)--(11.947,2.708);
\gpcolor{color=gp lt color border}
\gpsetlinetype{gp lt border}
\gpsetdashtype{gp dt solid}
\gpsetlinewidth{1.00}
\draw[gp path] (2.240,2.708)--(2.420,2.708);
\draw[gp path] (11.947,2.708)--(11.767,2.708);
\node[gp node right] at (2.056,2.708) {$10^{-5}$};
\draw[gp path] (2.240,3.450)--(2.330,3.450);
\draw[gp path] (11.947,3.450)--(11.857,3.450);
\draw[gp path] (2.240,3.884)--(2.330,3.884);
\draw[gp path] (11.947,3.884)--(11.857,3.884);
\draw[gp path] (2.240,4.192)--(2.330,4.192);
\draw[gp path] (11.947,4.192)--(11.857,4.192);
\draw[gp path] (2.240,4.431)--(2.330,4.431);
\draw[gp path] (11.947,4.431)--(11.857,4.431);
\draw[gp path] (2.240,4.627)--(2.330,4.627);
\draw[gp path] (11.947,4.627)--(11.857,4.627);
\draw[gp path] (2.240,4.792)--(2.330,4.792);
\draw[gp path] (11.947,4.792)--(11.857,4.792);
\draw[gp path] (2.240,4.935)--(2.330,4.935);
\draw[gp path] (11.947,4.935)--(11.857,4.935);
\draw[gp path] (2.240,5.061)--(2.330,5.061);
\draw[gp path] (11.947,5.061)--(11.857,5.061);
\gpcolor{color=gp lt color axes}
\gpsetlinetype{gp lt axes}
\gpsetdashtype{gp dt axes}
\gpsetlinewidth{0.50}
\draw[gp path] (2.240,5.174)--(11.947,5.174);
\gpcolor{color=gp lt color border}
\gpsetlinetype{gp lt border}
\gpsetdashtype{gp dt solid}
\gpsetlinewidth{1.00}
\draw[gp path] (2.240,5.174)--(2.420,5.174);
\draw[gp path] (11.947,5.174)--(11.767,5.174);
\node[gp node right] at (2.056,5.174) {$10^{-4}$};
\draw[gp path] (2.240,5.916)--(2.330,5.916);
\draw[gp path] (11.947,5.916)--(11.857,5.916);
\draw[gp path] (2.240,6.350)--(2.330,6.350);
\draw[gp path] (11.947,6.350)--(11.857,6.350);
\draw[gp path] (2.240,6.658)--(2.330,6.658);
\draw[gp path] (11.947,6.658)--(11.857,6.658);
\draw[gp path] (2.240,6.897)--(2.330,6.897);
\draw[gp path] (11.947,6.897)--(11.857,6.897);
\draw[gp path] (2.240,7.092)--(2.330,7.092);
\draw[gp path] (11.947,7.092)--(11.857,7.092);
\draw[gp path] (2.240,7.257)--(2.330,7.257);
\draw[gp path] (11.947,7.257)--(11.857,7.257);
\draw[gp path] (2.240,7.400)--(2.330,7.400);
\draw[gp path] (11.947,7.400)--(11.857,7.400);
\draw[gp path] (2.240,7.526)--(2.330,7.526);
\draw[gp path] (11.947,7.526)--(11.857,7.526);
\gpcolor{color=gp lt color axes}
\gpsetlinetype{gp lt axes}
\gpsetdashtype{gp dt axes}
\gpsetlinewidth{0.50}
\draw[gp path] (2.240,7.639)--(9.007,7.639);
\draw[gp path] (11.763,7.639)--(11.947,7.639);
\gpcolor{color=gp lt color border}
\gpsetlinetype{gp lt border}
\gpsetdashtype{gp dt solid}
\gpsetlinewidth{1.00}
\draw[gp path] (2.240,7.639)--(2.420,7.639);
\draw[gp path] (11.947,7.639)--(11.767,7.639);
\node[gp node right] at (2.056,7.639) {$10^{-3}$};
\draw[gp path] (2.240,8.381)--(2.330,8.381);
\draw[gp path] (11.947,8.381)--(11.857,8.381);
\draw[gp path] (2.240,0.985)--(2.240,1.075);
\draw[gp path] (2.240,8.381)--(2.240,8.291);
\draw[gp path] (2.764,0.985)--(2.764,1.075);
\draw[gp path] (2.764,8.381)--(2.764,8.291);
\gpcolor{color=gp lt color axes}
\gpsetlinetype{gp lt axes}
\gpsetdashtype{gp dt axes}
\gpsetlinewidth{0.50}
\draw[gp path] (3.013,0.985)--(3.013,8.381);
\gpcolor{color=gp lt color border}
\gpsetlinetype{gp lt border}
\gpsetdashtype{gp dt solid}
\gpsetlinewidth{1.00}
\draw[gp path] (3.013,0.985)--(3.013,1.165);
\draw[gp path] (3.013,8.381)--(3.013,8.201);
\node[gp node center] at (3.013,0.677) {$100$};
\draw[gp path] (3.787,0.985)--(3.787,1.075);
\draw[gp path] (3.787,8.381)--(3.787,8.291);
\draw[gp path] (4.809,0.985)--(4.809,1.075);
\draw[gp path] (4.809,8.381)--(4.809,8.291);
\draw[gp path] (5.334,0.985)--(5.334,1.075);
\draw[gp path] (5.334,8.381)--(5.334,8.291);
\gpcolor{color=gp lt color axes}
\gpsetlinetype{gp lt axes}
\gpsetdashtype{gp dt axes}
\gpsetlinewidth{0.50}
\draw[gp path] (5.583,0.985)--(5.583,8.381);
\gpcolor{color=gp lt color border}
\gpsetlinetype{gp lt border}
\gpsetdashtype{gp dt solid}
\gpsetlinewidth{1.00}
\draw[gp path] (5.583,0.985)--(5.583,1.165);
\draw[gp path] (5.583,8.381)--(5.583,8.201);
\node[gp node center] at (5.583,0.677) {$1000$};
\draw[gp path] (6.356,0.985)--(6.356,1.075);
\draw[gp path] (6.356,8.381)--(6.356,8.291);
\draw[gp path] (7.378,0.985)--(7.378,1.075);
\draw[gp path] (7.378,8.381)--(7.378,8.291);
\draw[gp path] (7.903,0.985)--(7.903,1.075);
\draw[gp path] (7.903,8.381)--(7.903,8.291);
\gpcolor{color=gp lt color axes}
\gpsetlinetype{gp lt axes}
\gpsetdashtype{gp dt axes}
\gpsetlinewidth{0.50}
\draw[gp path] (8.152,0.985)--(8.152,8.381);
\gpcolor{color=gp lt color border}
\gpsetlinetype{gp lt border}
\gpsetdashtype{gp dt solid}
\gpsetlinewidth{1.00}
\draw[gp path] (8.152,0.985)--(8.152,1.165);
\draw[gp path] (8.152,8.381)--(8.152,8.201);
\node[gp node center] at (8.152,0.677) {$10000$};
\draw[gp path] (8.925,0.985)--(8.925,1.075);
\draw[gp path] (8.925,8.381)--(8.925,8.291);
\draw[gp path] (9.948,0.985)--(9.948,1.075);
\draw[gp path] (9.948,8.381)--(9.948,8.291);
\draw[gp path] (10.472,0.985)--(10.472,1.075);
\draw[gp path] (10.472,8.381)--(10.472,8.291);
\gpcolor{color=gp lt color axes}
\gpsetlinetype{gp lt axes}
\gpsetdashtype{gp dt axes}
\gpsetlinewidth{0.50}
\draw[gp path] (10.721,0.985)--(10.721,6.969);
\draw[gp path] (10.721,8.201)--(10.721,8.381);
\gpcolor{color=gp lt color border}
\gpsetlinetype{gp lt border}
\gpsetdashtype{gp dt solid}
\gpsetlinewidth{1.00}
\draw[gp path] (10.721,0.985)--(10.721,1.165);
\draw[gp path] (10.721,8.381)--(10.721,8.201);
\node[gp node center] at (10.721,0.677) {$100000$};
\draw[gp path] (11.495,0.985)--(11.495,1.075);
\draw[gp path] (11.495,8.381)--(11.495,8.291);
\draw[gp path] (2.240,8.381)--(2.240,0.985)--(11.947,0.985)--(11.947,8.381)--cycle;
\node[gp node center,rotate=-270] at (0.246,4.683) {$|J(u)-J(u_h)|$};
\node[gp node center] at (7.093,0.215) {\#dofs primal};
\node[gp node right] at (10.479,8.047) {uniform};
\gpcolor{rgb color={0.580,0.000,0.827}}
\draw[gp path] (10.663,8.047)--(11.579,8.047);
\draw[gp path] (2.832,7.865)--(4.228,6.600)--(5.694,5.578)--(7.199,4.634)--(8.724,3.751)%
  --(10.261,2.879)--(11.802,1.942);
\gpcolor{color=gp lt color border}
\node[gp node right] at (10.479,7.239) {$\eta^{\varrho\varpi}_K$};
\gpcolor{rgb color={0.000,0.620,0.451}}
\draw[gp path] (10.663,7.239)--(11.579,7.239);
\draw[gp path] (2.832,7.865)--(3.552,6.644)--(4.301,5.561)--(5.060,4.924)--(5.815,4.122)%
  --(6.560,3.327)--(7.318,2.546)--(8.057,1.525);
\gpcolor{color=gp lt color border}
\node[gp node right] at (10.479,6.431) {$\eta^{\text{II},\ast}_K$};
\gpcolor{rgb color={0.337,0.706,0.914}}
\draw[gp path] (10.663,6.431)--(11.579,6.431);
\draw[gp path] (2.832,7.865)--(3.572,6.700)--(4.315,5.611)--(5.035,4.884)--(5.763,4.168)%
  --(6.502,3.352)--(7.257,2.561)--(8.005,1.513);
\gpcolor{color=gp lt color border}
\node[gp node right] at (10.479,5.623) {$\eta^{\varrho\tau}_K$};
\gpcolor{rgb color={0.902,0.624,0.000}}
\draw[gp path] (10.663,5.623)--(11.579,5.623);
\draw[gp path] (2.832,7.865)--(3.586,6.568)--(4.366,5.689)--(5.244,5.128)--(6.147,4.388)%
  --(7.047,3.883)--(7.959,3.009)--(8.902,2.146);
\gpcolor{color=gp lt color border}
\draw[gp path] (2.240,8.381)--(2.240,0.985)--(11.947,0.985)--(11.947,8.381)--cycle;
%% coordinates of the plot area
\gpdefrectangularnode{gp plot 1}{\pgfpoint{2.240cm}{0.985cm}}{\pgfpoint{11.947cm}{8.381cm}}
\end{tikzpicture}
%% gnuplot variables
}
  \\[1em]
  \caption{
    Performance plot; error in the quantity of interest $|J(u)-J(u_h)|$
    plotted over total number of primal degrees of freedoms for (a) the
    model problem on the unit square and (b) the model problem on
    the slit domain. Results for lowest-order, linear finite elements.
  }
  \label{fig:performanceplot}
\end{figure}
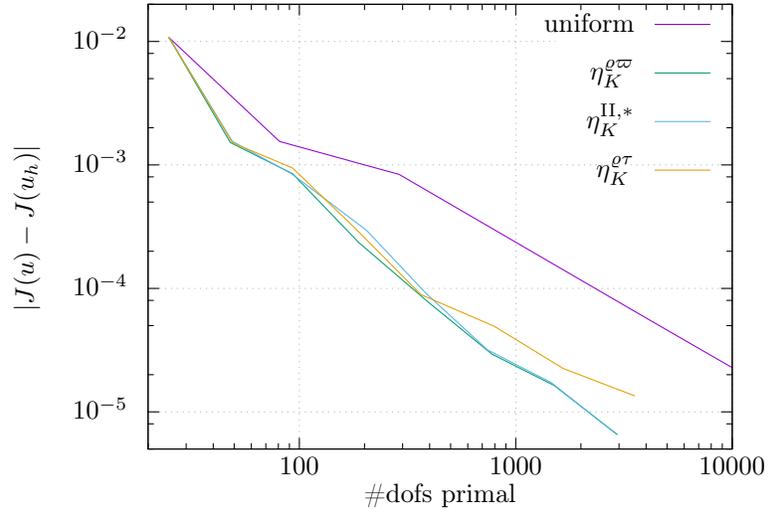
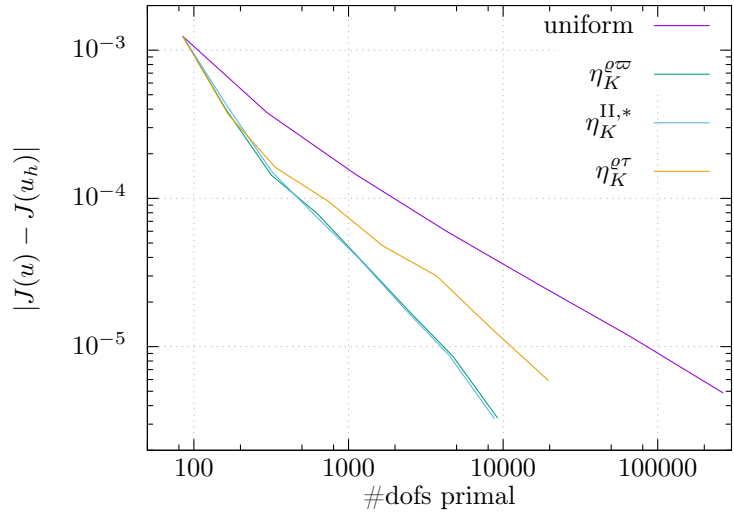

\begin{figure}[t!]
  \centering
  \subfloat[$\eta^{\varrho\varpi}$]{\includegraphics[height=3.9cm]{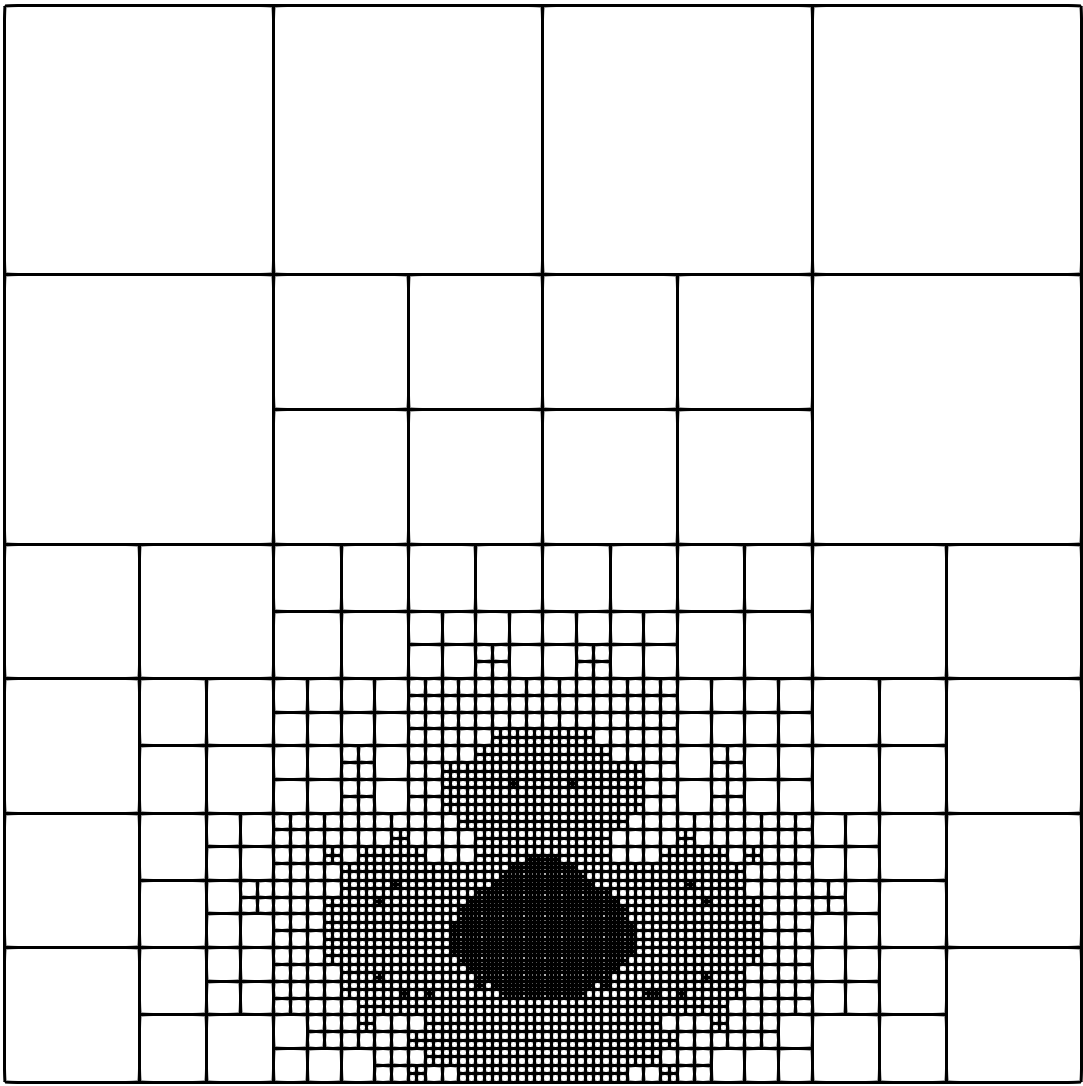}}
  \hspace{1em}
  \subfloat[$\eta^{\rm II,\ast}$]{\includegraphics[height=3.9cm]{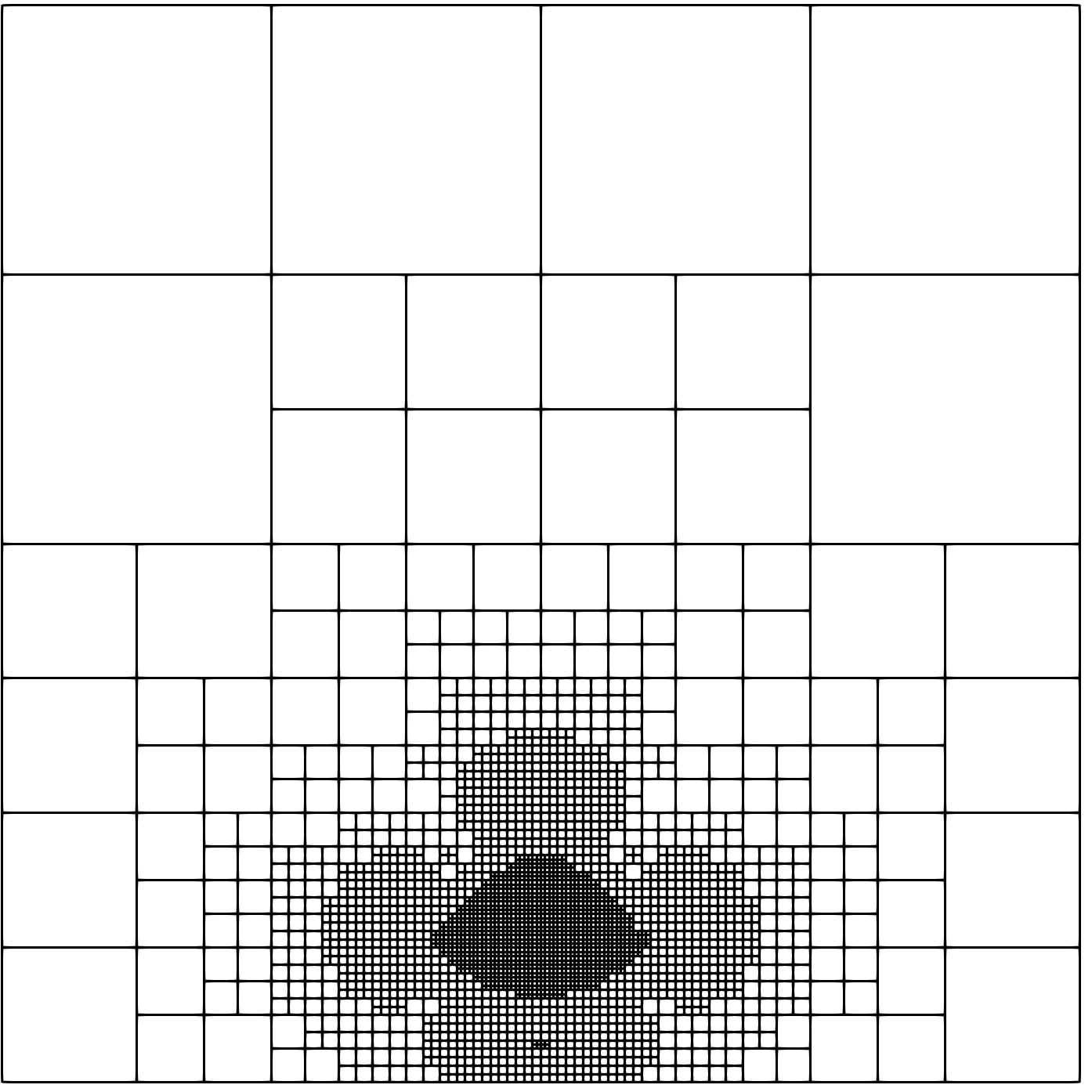}}
  \hspace{1em}
  \subfloat[$\eta^{\varrho\tau}$]{\includegraphics[height=3.9cm]{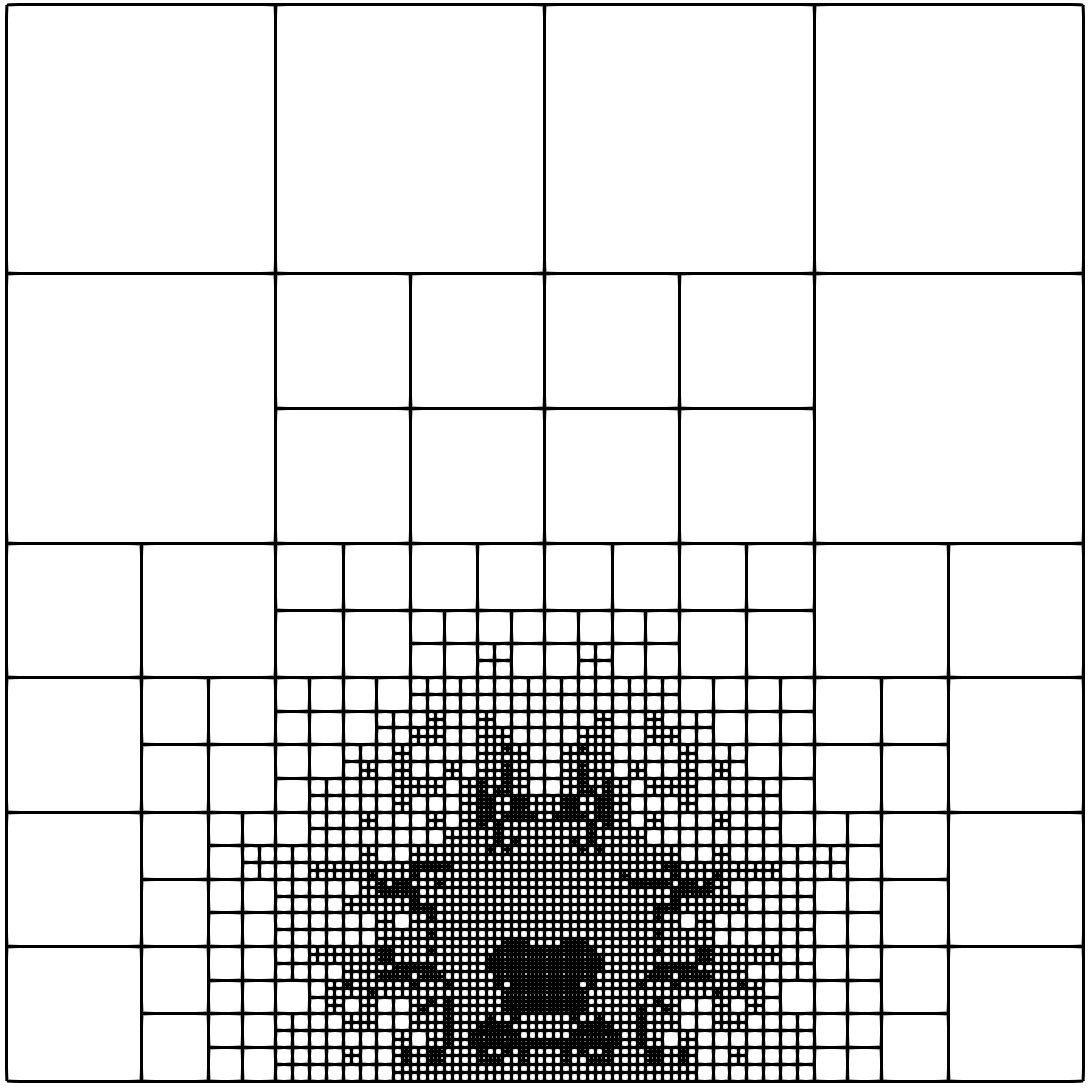}}
  \\[1em]
  \caption{
    Resulting locally refined meshes obtained by different local indicators
    for the test case defined on the unit square. The unbalanced,
    oscillatory nature of $\eta^{\varrho\tau}$ causes significant spurious refinement.}
  \label{fig:refinedmeshes}
\end{figure}

We conclude with testing the performance of the local indicators 
for adaptive refinement with lowest order Lagrange elements. 
Again, the local smoothed Dirac deltas are the goal functionals.
% both test cases with a local Dirac delta are used.
Figure~\ref{fig:performanceplot} shows a ``performance plot'' (error over 
number of primal degrees of freedom). All local indicators perform
qualitatively better than uniform refinement. In both test cases, the two
indicators $\eta^{\varrho\varpi}$ and $\eta^{\rm II,\ast}$ perform optimally,
i.e., with linear order of convergence. 
The indicator variant $\eta^{\varrho\tau}$, however, leads to
suboptimal asymptotic behavior. We attribute this to the oscillatory
behavior, i.e., to ``unbalanced'' local indicators 
(see Table~\ref{table:test2}). 
Figure~\ref{fig:refinedmeshes} displays the final locally refined meshes 
for visual comparison. 
The spurious refinement caused by the oscillatory nature 
of the indicator $\eta^{\varrho\tau}$ is clearly visible. 
% 
% A visual comparison of the final locally refined meshes
% can be found in Figure~\ref{fig:refinedmeshes}, which displays the spurious
% refinement caused by the oscillatory nature of both indicators. 
% 
% a comparison of the final locally refined meshes can be found.
% Figure~\ref{fig:performanceplot} further exemplifies this; 

\section{Conclusion}
\label{sec:conclusion}

The nature of our results is both experimental and theoretical. 

Equilibrated error estimators have been implemented with the software
library deal.II \cite{dealII85}.
As a theoretical foundation we have discussed 
the construction and well-posedness 
of local flux reconstruction problems 
over quadrilateral meshes with hanging nodes 
(see Lemma~\ref{lemma:einzigeslemma}). 
% {\matthias
Techniques inspired by commuting interpolators have been used, the bulk of
which have only been discussed for conforming simplicial meshes in the
literature so far.
% }%
%
Our computational experiments show uniformly bounded efficiency indices
for all tested polynomial orders. 
The experimental observation that the localized flux reconstruction leads
to competitive error estimates when compared to flux reconstruction via
mixed finite element methods is corroborated by the inverse inequality
\eqref{math:flux:globalvslocal}.
% {\matthias 
To the best of our knowledge,
% } 
this is the first time this
experimental observation has been traced to a rigorous estimate. We
emphasize that the efficiency of the equilibrated error
estimate can be understood best through convergence estimates for mixed
finite element methods. 

Our practical main results target the goal-oriented error estimation. The
proposed error estimator $\eta^{\varrho\varpi}$ shows efficiency indices in the
range $0.9-2.0$, depending on the regularity of the domain, and its
oscillation index is very well controlled by $1.2$. This distinguishes it
from the dual weighted residual method, which typically requires
sufficiently fine meshes and higher problem regularity.
Furthermore,
our local indicators $\eta^{\varrho\varpi}_{K}$
shows advantages over 
% we showed that our choice of local indicators is advantageous to 
the error indicators $\eta^{\varrho\tau}_{K}$, whose high
oscillation index points to their suboptimal marking in goal-oriented
adaptive finite element methods.
We therefore recommend the indicators $\eta^{\varrho\varpi}_{K}$ 
for adaptive marking and error estimation in the quantity of interest 
at the cost of numerically solving both the primal and the dual problem 
and performing a localized flux reconstruction for both. 

% 
% We draw the following main conclusions from our numerical experiments:
% \begin{enumerate}[(i)]
%  \item Equilibrated error estimates using the local flux reconstruction 
%  are practically not much worse than the ones obtained 
%  with the globally reconstructed flux. 
%  The presence of re-entrant corners is generally adversive 
%  to their efficiency. 
% \end{enumerate}

\bibliographystyle{amsplain}
\bibliography{references}

\end{document}